\documentclass[12pt,letterpaper]{article}

\newcommand{\Keywords}[1]{\par\noindent{\small{\bf Keywords\/}: #1}}
\newcommand{\Class}[1]{\par\noindent{\small{\bf Mathematics Subjects Classification (2010)\/}: #1}}

\usepackage{amsthm, amsmath, amssymb, amsfonts, graphicx, epsfig}
\usepackage{accents}
\usepackage{dsfont}

\usepackage[T1]{fontenc}
\usepackage{hyperref}
\usepackage{bbm}
\usepackage{mathrsfs}
\usepackage{ulem}
\usepackage{cancel}

\usepackage{epstopdf}

\usepackage{graphicx}
\usepackage[font=sl,labelfont=bf]{caption}
\usepackage{subcaption}

\usepackage{enumerate}
\usepackage{enumitem}

\usepackage[usenames]{color}

\usepackage{url}
\makeatletter\def\url@leostyle{%
 \@ifundefined{selectfont}{\def\UrlFont{\sf}}{\def\UrlFont{\scriptsize\ttfamily}}} \makeatother

\usepackage[framemethod=default]{mdframed}

\usepackage[margin=1.01in, letterpaper]{geometry}
\newtheorem{theorem}{Theorem}
\newtheorem{proposition}[theorem]{Proposition}
\newtheorem{lemma}[theorem]{Lemma}

\theoremstyle{definition}
\newtheorem{definition}[theorem]{Definition}
\newtheorem{example}[theorem]{Example}
\theoremstyle{remark}
\newtheorem{remark}[theorem]{Remark}

\numberwithin{equation}{section}
\numberwithin{theorem}{section}

\definecolor{Red}{rgb}{0.9,0,0.0}
\definecolor{Blue}{rgb}{0,0.0,1.0}

\def\si#1{}
\def\ni#1{}

\DeclareMathAlphabet\mathbfcal{OMS}{cmsy}{b}{n}

\def\mypart#1{}

\usepackage{t1enc}	
\newcommand{\R}{\mathbb R}
\DeclareFontFamily{U}{mathx}{\hyphenchar\font45}
\DeclareFontShape{U}{mathx}{m}{n}{
	<5> <6> <7> <8> <9> <10>
	<10.95> <12> <14.4> <17.28> <20.74> <24.88>
	mathx10
}{}
\DeclareSymbolFont{mathx}{U}{mathx}{m}{n}
\DeclareFontSubstitution{U}{mathx}{m}{n}
\DeclareMathSymbol{\bigtimes}{1}{mathx}{"91}
\DeclareMathAccent{\widebar}{0}{mathx}{"73}

\def\cA{\mathcal{A}}
\def\cB{\mathcal{B}}

\def\cE{\mathcal{E}}
\def\cF{\mathcal{F}}
\def\cG{\mathcal{G}}
\def\cH{\mathcal{H}}
\def\cI{\mathcal{I}}
\def\cJ{\mathcal{J}}

\def\cN{\mathcal{N}}

\def\cP{\mathcal{P}}

\def\cU{\mathcal{U}}

\def\cX{\mathcal{X}}
\def\cY{\mathcal{Y}}

\def\bE{\mathbb{E}}
\def\bF{\mathbb{F}}
\def\bG{\mathbb{G}}
\def\bH{\mathbb{H}}

\def\bN{\mathbb{N}}

\def\bP{\mathbb{P}}

\def\bR{\mathbb{R}}

\newcommand{\1}{\mathbbm{1}}            
\newcommand{\set}[1]{{\{#1\}}}            

\newcommand{\wh}[1]{\widehat{#1}}
\newcommand{\wt}[1]{\widetilde{#1}}

 \def\hat{\widehat}
 
\def\tilde{\widetilde}

 \def\F{{\cal F}}

\newtheorem{thm}{Theorem}[section]
\newtheorem{lem}{Lemma}[section]
\newtheorem{pro}{Proposition}[section]
\newtheorem{cor}{Corollary}[section]
\newtheorem{rem}{Remark}[section]
\newtheorem{ex}{Example}[section]
\newtheorem{defi}{Definition}[section]

  \newcommand{\be}{\begin{equation}}
\newcommand{\ee}{\end{equation}}
\newcommand{\bde}{\begin{displaymath}}
\newcommand{\ede}{\end{displaymath}}
\newcommand{\beq}{\begin{eqnarray*}}
\newcommand{\eeq}{\end{eqnarray*}}
\newcommand{\beqa}{\begin{eqnarray}}
\newcommand{\eeqa}{\end{eqnarray}}
\newcommand{\bel }{\left{\begin{array}{ll}}}
\newcommand{\eel}{\cr \end{array} \right.}
\newcommand{\bd}{\begin{defi}}
\newcommand{\ed}{\end{defi}}
\newcommand{\brem }{\begin{rem} \rm }
\newcommand{\erem }{\end{rem}}
\newcommand{\bex}{\begin{ex} \rm }
\newcommand{\eex}{\end{ex}}
\newcommand{\begth}{\begin{thm}}
\newcommand{\eeth}{\end{thm}}
\newcommand{\bl}{\begin{lem}}
\newcommand{\el}{\end{lem}}
\newcommand{\bp}{\begin{pro}}
\newcommand{\ep}{\end{pro}}
\newcommand{\bcor}{\begin{cor}}
\newcommand{\ecor}{\end{cor}}
\newcommand{\lab }{\label }

\usepackage{fancyhdr}
\usepackage{lastpage}
\title{Generalized Multivariate Hawkes Processes}

\author{Tomasz R. Bielecki
\\ Department of Applied Mathematics \\
 Illinois Institute of Technology \\
 Chicago, IL 60616, USA \\ \\
Jacek Jakubowski
\\  Institute     of Mathematics  \\ University of Warsaw
\\ Warszawa, Poland \\ \\
 Mariusz Niew\k{e}g\l owski \\
Faculty of Mathematics and Information Science
\\ Warsaw University of Technology
\\ Warszawa, Poland}

\date{\vskip 30 pt \today \vskip 25 pt}

\begin{document}
\maketitle
\thispagestyle{empty}	

\begin{abstract}
This work contributes to the theory and applications of Hawkes processes. We introduce and examine a new class of Hawkes processes that we call generalized Hawkes processes, and their special subclass -- the generalized multivariate Hawkes processes (GMHPs). GMHPs are multivariate marked point processes that add an important feature to the family of the (classical) multivariate Hawkes processes: they allow for explicit modelling of simultaneous occurrence of  excitation events coming from different sources, i.e. caused by different coordinates of the multivariate process. We study the issue of existence of a generalized Hawkes process, and we provide a construction of a specific generalized multivariate Hawkes process. We investigate Markovian aspects of GMHPs, and we indicate some plausible important applications of GMHPs.
\vskip 20 pt
\vskip 20 pt
\Class{60G55, $\,$60H99}
\end{abstract}


\tableofcontents

\section{Introduction}

A very interesting and important class of stochastic processes was introduced by Alan Hawkes in \cite{Hawkes1971a,Hawkes1971}. These processes, called now  Hawkes processes, are meant to model self-exciting and mutually-exciting random phenomena that evolve in time. The self-exciting phenomena are modeled as univariate Hawkes processes, and the mutually-exciting phenomena are modeled as multivariate Hawkes processes. Hawkes processes belong to the family  of  marked point processes, and, of course, a univariate Hawkes process is just a special case of the multivariate one.

In this paper, which originates from Chapter 11 of \cite{BJN-buczek}, we define and study generalized multivariate Hawkes processes (GMHPs).  These processes constitute  a subclass of the family generalized Hawkes processes defined in this paper as well.  In addition, we provide a novel construction of a generalized multivariate Hawkes process.

GMHPs are multivariate marked point processes that add an important feature to the family of the (classical) multivariate Hawkes processes: they allow for explicit modelling of simultaneous occurrence of  excitation events coming from different sources, i.e. caused by different coordinates of the multivariate process. The importance of this feature is rather intuitive, and it will be illustrated  in Section \ref{sec:app}. In this regard, GMHPs differ from the  multivariate Hawkes processes that were studied in Bremaud and Massouli \cite{BreMas1996}  and Liniger \cite{lin2009}.

We need to stress that we limit ourselves here to the case of linear GMHPs, that are counterpart of the linear classical Hawkes processes. That is to say, we do not study here what would be {a} counterpart of the nonlinear classical Hawkes processes. We refer to e.g. Chapter 1 in \cite{Zhu2013} for comparison of linear and nonlinear Hawkes processes. We also note that the generalized Hawkes processes introduced here should not be confused with those studied in \cite{Anca2011}. In particular, we do not introduce any additional random factors, such as Brownian motions, into the compensators of the multivariate marked point process $N$ showing in the Definition \ref{def:multiHawkesi} below.

We also need to stress that we are not concerned in this study with stationarity and spectral properties of the GMHPs. This is the reason why in the definition of the Hawkes kernel $\kappa$, of the generalized Hawkes process, we use integration over the interval $(0,t)$ rather than integration over $(-\infty,t)$. Please see also Remark \ref{rem:(0,t)} in this regard.

The paper is organized as follows. In Section \ref{sec:GenGHawkes} we define, prove existence of and provide some discussion of a generalized Hawkes process. Section \ref{GMHP} is devoted to study of the main object of this paper, namely the generalized multivariate Hawkes process. In Section \ref{constr} we provide a mathematical construction of and computational pseudo-algorithm for simulation of a generalized multivariate Hawkes process with deterministic kernels $\eta$ and $f$ (cf. \eqref{eq:kappa-G}). Markovian aspects of a generalized multivariate Hawkes process are discussed in Section \ref{Ex:MarkovHawkes}. Section \ref{sec:app} contains
 a brief description of possible applications of generalized multivariate Hawkes processes {in seismology, epidemiology and finance.} Finally, in the Appendix, we provide some needed technical results.

\medskip
\noindent
In this paper we use various  concepts and results from stochastic analysis. For a comprehensive study of these concepts and results we refer to e.g. \cite{HeWanYan1992}, \cite{LasBra1995} and \cite{js1987}.

\section{Generalized Hawkes process}\label{sec:GenGHawkes}

Let $(\Omega,\cF,\bP)$  be  a probability space and $(\mathcal{X},\mathbfcal{X} )$ be a Borel space.
We take $\partial$ to be a point external to $\mathcal{X}$, and we let $\mathcal{X}^\partial := \mathcal{X} \cup \partial$. On $(\Omega,\cF,\bP)$ we consider a \textit{marked point process} $N$ with  mark space $\mathcal{X}$, that is, a sequence of random elements
\begin{equation}\label{eq:Ngen-G}
N=((T_n, X_n))_{n \geq 1},
\end{equation}
where for each $n$:
\begin{enumerate}
	\item $T_n$ is a random variable with values in $(0,\infty]$,
	\item $X_n$ is a random variable  with values in $\mathcal{X}^\partial$,
	\item $
	T_n \leq T_{n+1}$, and  if $T_n < + \infty$ then  $
	T_n < T_{n+1}$,
	\item  $ X_n = \partial $ iff $T_n = \infty$.
\end{enumerate}
The explosion time of $N$, say $T_\infty$, is defined as
\[
T_\infty := \lim_{ n \rightarrow \infty } T_n.
\]
Following the typical techniques used in the theory of Marked Point Processes (MPPs), in particular following Section 1.3 in  \cite{LasBra1995}, we associate with the process $N$ an integer-valued random measure on $(\mathbb{R}_+ \times \mathcal{X}, \mathcal{B}(\mathbb{R}_+) \otimes \mathbfcal{X} )$, also denoted by $N$ and defined as
\begin{equation}\label{eq:NH-G}
N(dt,dx) := \sum_{n \geq 1 } \delta_{(T_n, X_n)} (dt, dx) \1_\set{T_n < \infty},
\end{equation}
so that
\[
N((0,t],A)=\sum_{n\geq 1} \1_{\{ T_n \leq t\}}\1_{\{X_n \in A \}},
\]
where $A\in\mathbfcal{X}$.

Let $\bF^N $ be the  natural filtration of $N$,  so ${\mathbb F}^N:=(\F^N_t,\ t\geq 0)$, where $\F^N_t$
is the $\bP$--completed $\sigma$--field $\sigma(N((s,r]\times A)\, :\, 0\leq s<r\leq t,\ A\in \cX),\ t\geq 0$. In view of Theorem 2.2.4 in \cite{LasBra1995} the filtration ${\mathbb F}^{N}$ satisfies the usual conditions.
Moreover,  $N$ is $\bF^N$--optional, so, using Proposition 4.1.1 in \cite{LasBra1995} we conclude that $T_n$'s are $\bF^N$--stopping times and $X_n$ are $\cF_{T_n}$-measurable.
In what follows we denote by $\cP$  the $\bF^N$-predictable $\sigma$-field.

We recall that for a given filtration $\bF$ a stochastic process $X: \Omega \times [0,\infty ) \rightarrow \bR$ is said to be  $\mathbb{F}$-predictable if it is measurable with respect to the predictable sigma field $\cP^\bF $ on $\Omega \times [0,\infty )$, which is generated by $\bF$-adapted processes whose paths are continuous (equivalently left-continuous, with the left limit at $t=0$ defined as the value of the path at $t=0$)  functions of time variable. More generally, a function $X: \Omega \times [0,\infty ) \times \mathcal{X} \rightarrow \bR$ is said to be $\mathbb{F}$-predictable function if it is measurable with respect to the sigma field $\cP^\bF(\mathbfcal{X}):= \cP^\bF \otimes \mathbfcal{X}$ on $\Omega \times [0,\infty) \times \mathcal{X}$.
The sigma field $\cP^\bF(\mathbfcal{X})$  is generated by the sets $A \times \set{0} \times \mathbfcal{X} $  where $A \in \cF_0$ and the sets of the form $B \times (s,t] \times D$ where $0 < s \leq t$, $B \in \cF_s$  and $D \in \mathbfcal{X}$.

We now consider a random measure $\nu$ on $(\mathbb{R}_+ \times \mathcal{X}, \mathcal{B}(\mathbb{R}_+) \otimes \mathbfcal{X} )$ defined as
\begin{equation}\label{eq:nu-Hawkes-G}
\nu(\omega,dt,dy):=\1_{]\!] 0, T_\infty(\omega) [\![}( t)\kappa(\omega,t,dy)dt,
\end{equation}
where, for $A\in \mathbfcal{X}$,
\begin{equation}\label{eq:kappa-G}
\kappa(t,A)=\eta(t,A)+\int_{(0, t)\times \mathcal{X} } f(t,s,x,A)N(ds,dx),
\end{equation}
$\eta$ is a finite  kernel from $(\Omega \times [0,\infty), \mathcal{P})$ to $(\mathcal{X},\mathbfcal{X})$, and $f$ is a kernel from $(\Omega \times \R_+ \times \R_+ \times \mathcal{X}, \cF \otimes \cB(\R_+)\otimes \cB(\R_+)\otimes \mathbfcal{X} )$ to $(\mathcal{X},\mathbfcal{X})$.\footnote{See Appendix A.2 in Last and Brandt \cite{LasBra1995} for the definition of the kernel.}
We assume also  that $f$ is a kernel satisfying:
\begin{enumerate}
	\item
	$f(t,s,x,A)=0$ for $s\geq t$,
	\item  $\theta$ defined as
	\[
	\theta(t,A):=\int_{(0, t)\times \mathbfcal{X}} f(t,s,x,A)N(ds,dx),\quad t\geq 0,\ A\in \mathbfcal{X},
	\]
	is a kernel from $(\Omega \times [0,\infty), \mathcal{P})$ to $(\mathcal{X},\mathbfcal{X})$, which is finite for $t < T_\infty$.
\end{enumerate}
 Clearly, we have
\begin{equation}\label{eq:int-f-N}
\theta(t,A)= \sum_{n:\ T_n < t} f(t,T_n,X_n,A).
\end{equation}
 Note that $\kappa(t,\mathcal{X})$ is finite for any $t < T_\infty$. We additionally assume that $\kappa(t,\mathcal{X})>0$ for all $t\geq0$ and that the integral $ \int_{[0,t]}\kappa(s,A)ds$ is finite for any $A\in \mathbfcal{X} $ and any $t < T_\infty$. This last assumption is satisfied under mild boundedness conditions imposed  on $\eta$ and~$f$.

Note that the process $\nu([0,\cdot ],A)=\int_{[0,\cdot ]}\1_{]\!] 0, T_\infty(\omega) [\![}( s)\kappa(s,A)ds$ is continuous for any set $A\in \mathbfcal{X} $ and thus it is $\bF^N$--predictable.  Consequently,  $\nu$ is a $\bF^N$--predictable random measure.

\medskip \noindent

	\noindent Before we proceed we recall that for a given filtration $\bF$  the random measure $\nu$ is said to be $\bF$-compensator of a random measure $N$ if it is $\bF$-predictable random measure such that it holds
	\[
	\bE \int_0^\infty \int_{E^\Delta} F(v,x) N(dv,dx) =
	\bE \int_0^\infty \int_{E^\Delta} F(v,x) \nu(dv,dx)
	\]
	for every non-negative $\bF$-predictable function $F: \Omega \times [0,\infty) \times \mathcal{X} \rightarrow \bR$.

We are ready to state the underlying definition in this  paper.

\begin{definition}\label{def:genHawkes} Let $N$ be the marked point process  introduced in \eqref{eq:Ngen-G} with
	the corresponding random measure  $N$ defined in \eqref{eq:NH-G}.
	We call $N$ a \textit{generalized  Hawkes process}  on $(\Omega, \cF, \bP)$,  if
	the $(\bF^N,\bP)$--compensator of  $N$, say $\nu$,  is of the form \eqref{eq:nu-Hawkes-G}. The kernel $\kappa$ is called the  \textit{Hawkes kernel} for $N$.
	\qed
	\end{definition}

\begin{remark}\label{rem:(0,t)}
We note that  in our definition of the generalized Hawkes process the integral in \eqref{eq:kappa-G} is taken over the interval $(0,t)$. In the definition of the classic Hawkes process, the corresponding integral is taken over the interval $(-\infty,t)$; see eg. \cite{ELL}. The $``(0,t)"$ convention is used by several authors, though, in many applications of classical Hawkes processes (such as in Example \ref{ex:clasicHawkes}) that do not regard stationarity and spectral properties of these processes.  We use this convention here  since we are not considering stationarity and spectral properties of the generalized Hawkes processes.
\end{remark}

\begin{remark}

\noindent (i) Recall that the compensator of a random measure is unique (up to equivalence). Thus, the compensator $\nu$ of $N$ is unique. However, the representation \eqref{eq:nu-Hawkes-G}-\eqref{eq:kappa-G} is not unique, by any means, in general. For any given $\eta$ and $f$ in the representation \eqref{eq:nu-Hawkes-G}-\eqref{eq:kappa-G}, one can always find $\widetilde \eta\ne \eta$ and $\widetilde f\ne f$ such that
\begin{equation}\label{eq:wt-kappa-G}
\kappa(t,dy)=\wt \eta(t,dy)+\int_{(0, t)\times \mathcal{X} } \wt f(t,s,x,dy)N(ds,dx).
\end{equation}
 \noindent (ii) With a slight abuse of terminology we refer to $\kappa$ as to the Hawkes intensity kernel of $N$.
  Accordingly, we refer to the quantity $\kappa(t,A)$ as to the intensity at time $t$ of the event regarding process $N$ and amounting  to the marks of $N$ taking values in the set $A$, or, for short, as to the intensity at time $t$ of marks of $N$ taking values in $A$. \qed
\end{remark}

\begin{remark}\label{rem:law-of-N-under-P}
	Since $\bF^N_0$ is a completed trivial $\sigma$-field, then it is a consequence of Theorem 3.6 in \cite{Jac1975} that the compensator $\nu$ determines the law of $N$ under $\bP$, and, consequently, the Hawkes kernel $\kappa$  determines the law of $N$ under $\bP$. 	\qed
\end{remark}

\subsection{Existence of a generalized Hawkes process}\label{sec:existence}
 We will now demonstrate that for an arbitrary measure $\nu $ of the form \eqref{eq:nu-Hawkes-G}  there exists a Hawkes process having $\nu$ as $\bF^N$--compensator. Towards this end we will consider the underlying canonical space.
Specifically, we take $(\Omega, \mathcal{F})$ to be the canonical space of multivariate marked point processes
with marks taking values in $\mathcal{X}^\partial$ . That is, $\Omega $ consists of elements $\omega = ((t_n , x_n))_{n \geq 1 }, $
satisfying  $(t_n , x_n)\in (0,\infty]\times \mathcal{X}^\partial$ and
\begin{align*}
& t_n \leq t_{n+1};
\\
& \textnormal{if } t_{n} <\infty, \textnormal{ then }  t_n < t_{n+1};
\\
&  t_{n} =\infty \textnormal{ iff }   x_n = \partial.
\end{align*}
The $\sigma$--field $\mathcal{F}$ is defined to be the smallest $\sigma$--field on $\Omega$ such that
the mappings $T_n:\Omega \rightarrow ([0,\infty], \mathcal{B}[0,\infty] )$, $X_n:\Omega\rightarrow (\mathcal{X}^\partial, \mathbfcal{X}^\partial) $ defined by
\[
T_n(\omega) := t_n, \quad X_n(\omega) := x_n
\]
are measurable for every $n$.

Note that the canonical space introduced above agrees with the definition of canonical space considered in \cite{LasBra1995} (see Remark 2.2.5 therein).
On this space we  denote by $N$ a sequence of measurable mappings
\begin{equation}\label{eq:N}
N=((T_n, X_n))_{n \geq 1},
\end{equation}
Clearly, these mappings satisfy
\begin{enumerate}
	\item $
	T_n \leq T_{n+1}$, and  if $T_n < + \infty$ then  $
	T_n < T_{n+1}$,
	\item  $ X_n = \partial $ iff $T_n = \infty$.
\end{enumerate}
We call such $N$ a canonical mapping.

The following result  provides  the existence of  a probability measure $\bP_\nu$
on  $(\Omega,\cF)$ such that the canonical mapping  $N$ becomes a generalized Hawkes process with a given Hawkes kernel $\kappa$, which in a unique way determines the compensator~$\nu$.

\begin{theorem}\label{prop:main} Consider the canonical space $(\Omega,\cF)$ and the canonical mapping $N$ given by \eqref{eq:N}. Let measures $N$ and $\nu$ be associated with this canonical mapping through \eqref{eq:NH-G} and \eqref{eq:nu-Hawkes-G}--\eqref{eq:kappa-G}, respectively. Then, there exists a unique probability measure $\bP_\nu$ on $(\Omega,\cF)$, such that the measure $\nu$ is an $({\mathbb F}^N,\bP_\nu)$--compensator of $N$.
	So, $N$ is a generalized  multivariate Hawkes process on $(\Omega,\cF, \bP_\nu)$.
\end{theorem}

\begin{proof}

	We will use Theorem  8.2.1 in \cite{LasBra1995} with
	$\textbf{X}=\mathcal{X}$,
	$\varphi=\omega$,
	and with
	\begin{equation}\label{eq:balphadef}
	\bar \alpha(\omega,dt):=\nu(\omega,dt,\mathcal{X}) =\1_{]\!] 0, T_\infty(\omega) [\![}( t)\kappa(\omega,t,\mathcal{X})dt,
	\end{equation}
	from which we will conclude the assertion of theorem.
	Towards this end, we will verify that all assumptions of  the said theorem are satisfied in the present case.
	As already observed, the random measure $\nu$ is $\bF^N$--predictable. Next, let  us fix $\omega \in \Omega$. Given \eqref{eq:balphadef} we see that
	$\bar \alpha$ satisfies the following equalities
	\[
	\bar \alpha(\omega,\set{0}) = 0, \qquad \bar \alpha(\omega,\set{t}) = 0 \leq 1, \quad t\geq 0,
	\]
	which correspond to conditions (4.2.6) and (4.2.7) in \cite{LasBra1995}, respectively.
	It remains to show that condition (4.2.8) holds as well, that is
	\begin{equation}\label{eq:balpha}
	\bar \alpha( \omega,[\![ \pi_\infty (\omega), \infty [\![ ) =0,
	\end{equation}
	where
	\[
	\pi_\infty (\omega):= \inf \set{ t\geq 0: \bar \alpha(\omega,(0,t]) = \infty }.
	\]
	To see this, we first note that  \eqref{eq:balphadef}  implies
	\[
	\bar \alpha( \omega,[\![ T_\infty (\omega), \infty [\![ ) =0.
	\]
	Thus it suffices to show that $\pi_\infty (\omega)\geq T_\infty (\omega)$.
	By definition of $\bar \alpha$ we can write
	\[
	\bar \alpha(\omega,(0,t])
	=
	\begin{cases}
	\int_0^t
	\kappa(\omega,s,\mathcal{X})
	ds, &\ t < T_\infty(\omega), \\
	\int_0^{ T_\infty(\omega) }
	\kappa(\omega,s,\mathcal{X}) ds, & \ t \geq   T_\infty(\omega).
	\end{cases}
	\]
	If $T_\infty(\omega) = \infty$, then we clearly have  $\pi_\infty(\omega) = \infty = T_\infty(\omega)$.
	
	Next, if $T_\infty(\omega) < \infty$, then $\lim_{t \uparrow T_\infty(\omega)} \bar \alpha(\omega,(0,t]) =a $.
	We need to consider two cases now: $a=\infty$ and $a<\infty$.
	
	If $a= \infty $, then $\bar \alpha(\omega,(0,t]) = \infty$ for $t\geq T_\infty(\omega)$, and,
	$\bar \alpha(\omega,(0,t])<\infty$ for $t<  T_\infty(\omega)$  in view of our assumptions imposed on $\kappa$ in the beginning of this section. 	This implies that $\pi_\infty(\omega) = T_\infty(\omega)$.
	
	If $a < \infty$, then $\bar \alpha(\omega,(0,t]) =a<\infty$ for $t\geq T_\infty(\omega)$, hence $\pi_\infty(\omega) =\infty\geq T_\infty(\omega)$.
	Thus,
	$\pi_\infty(\omega) \geq T_\infty(\omega)$, which implies that \eqref{eq:balpha} holds.
	
	Since $\omega$ was arbitrary, we conclude that for all $\omega\in \Omega$ conditions  (4.2.6)-(4.2.8) in \cite{LasBra1995} are satisfied.
	So, applying Theorem 8.2.1 in \cite{LasBra1995} with  $\beta=\nu$, we obtain that there exists a unique probability measure $\bP_\nu$ such that  $\nu$ is a ${\mathbb F}^N$--compensator of $N$ under $\bP_\nu$.
\end{proof}

\subsection{Cluster interpretation of the generalized  Hawkes processes}\label{sec:cluster}

The classical Hawkes processes are conveniently interpreted, or represented, in terms of so called clusters. This kind of representation is sometimes called immigration and birth representation. We refer to \cite{HawOak1974} and \cite{Lau2015}.

Generalized Hawkes processes also admit cluster representation. The dynamics of cluster centers, or the immigrants, is directed by $\eta$. Specifically, $\eta(t,A)$ is the time-$t$ intensity of arrivals of immigrants with marks belonging to set $A$. The dynamics of the off-springs is directed by $f$. Specifically, $f(t,s,x,A)$ represents the time-$t$ intensity of births of offsprings with marks in set $A$ of either an immigrant with mark $x$ who arrived at time $s$, or of an offspring with mark $x$ who was born at time $s$.

The cluster interpretation will be exploited in a follow-up work for asymptotic analysis of generalized Hawkes processes.

\section{Generalized multivariate Hawkes process}\label{GMHP}

We now introduce the concept of a generalized multivariate Hawkes
process, which is a particular case of the concept of a generalized Hawkes process.

\subsection{Definition}
We first construct an appropriate mark space. Specifically, we fix an integer $d\geq 1$ and we let $(E_i, \cE_i)$, $i=1,\ldots,d$,  be  some non-empty Borel spaces,
and  $\Delta $ be a dummy mark, the meaning of which will be explained below.
	Very often, in practical modelling,  spaces $E_i$ are discrete.
The instrumental rationale for considering a discrete mark space is that in most of the applications of the Hawkes processes that we are familiar with and/or we can imagine, a discrete mark space is sufficient to account for the intended features of the modeled phenomenon.

We set $E^\Delta_i := E_i \cup \Delta$, and we denote by $\cE^\Delta_i$ the sigma algebra on $E^\Delta_i$ generated by $\cE_i$. Then, we define a mark space, say $E^\Delta$,  as
\begin{equation}\label{eq:marks}
E^\Delta := E^\Delta_1 \times E^\Delta_2 \times \ldots \times E^\Delta_d
\setminus (\Delta, \Delta, \ldots, \Delta ).
\end{equation}
By $\cE^\Delta$ we denote a trace sigma algebra of $\otimes_{i=1}^d  \cE^\Delta_i $ on $E^\Delta$, i.e.
\[
\cE^\Delta := \Big\{ A \cap E^\Delta : A \in \otimes_{i=1}^d  \cE^\Delta_i
\Big\}.
\]
Moreover, denoting by  $\partial_i$ the point which is external to  $E^\Delta_i$, we define
$E^\partial_i := E_i^\Delta \cup \{\partial_i\}$,
and we denote by $\cE^\partial_i$
the sigma algebra generated by $\cE_i$ and $\{\partial_i\}$.
Analogously we define
\[
 E^\partial := E^\Delta  \cup \partial,
\]
where $\partial= (\partial ^1, \ldots, \partial^d)$
is a point external to  $ E^\Delta_1 \times E^\Delta_2 \times \ldots \times E^\Delta_d$ and by $\cE^\partial$
 we denote the sigma field generated by  $\cE^\Delta$ and $\set{\partial}$.

\begin{definition}\label{def:multiHawkesi}
A generalized Hawkes process $N=((T_n, X_n))_{n \geq 1}$ with the mark space $\cX=E^\Delta$  given by \eqref{eq:marks}, and with  $\mathcal{X}^\partial = E^\partial $, is called a  \textit{generalized multivariate Hawkes process (of dimension $d$)}. \qed
\end{definition}

Note that a necessary condition for generalized Hawkes processes to feature the self-excitation  and mutual-excitation is that $
f\neq 0$.  We refer to Example \ref{commonevents} for interpretation of the components $\eta$ and $f$ of the kernel $\kappa$ in case of a generalized multivariate Hawkes process.

 We interpret  $T_n\in (0,\infty)$ and $X_n \in E^\Delta$ as the event times of $N$ and as the corresponding mark values, respectively. Thus, if $T_n < \infty$ we have\footnote{Note that here  $d$ is the number of components in $X_n$, and $n$ is the index of the $n-th$ element in the sequence $( X_n)_{n \geq 1}$. }
 \[
 X_n =(X^i_n,i=1,2,\ldots, d
 ) , \quad \text{where} \quad
 X^i_n
 \in E^\Delta_i.
 \]
Also, we interpret $X^i$ as the marks associated with $i$-th coordinate of $N$ (cf.  Definition \ref{def:i-th}). With this interpretation, the equality   $X^i_n(\omega)=\Delta$ means that there  is no event taking place with regard to the
$i$-th coordinate of $N$ at the (general) event time $T_n(\omega)$.
 In other words, no event occurs with respect to the $i$-th coordinate of $N$ at time $T_n(\omega)$.

\begin{definition}\label{common-times}
 We say that $T_n(\omega)$ is a common event time for a multivariate Hawkes process $N$ if there exist $i$ and $j$, $i\ne j$, such that $X^i_n(\omega) \in E_i$ and  $X^j_n(\omega) \in E_j$. We say that process $N$ admits common event times if
 \[
 \bP \Big( \omega \in \Omega : \exists n \text{ such that } T_n(\omega) \textnormal{ is a common event time }\Big) > 0
 \]
 Otherwise we say that process $N$ admits no common event times. \qed
 \end{definition}

Definition \ref{common-times} generalizes that in Bremaud and Massouli \cite{BreMas1996}  and Liniger \cite{lin2009}.
In particular, with regard to the concepts of multivariate Hawkes processes studied in  Liniger \cite{lin2009}, the genuine multivariate Hawkes processes \cite{lin2009} admits no common event times, whereas  in the case of
pseudo-multivariate Hawkes process \cite{lin2009}  all event times are common.

\subsection{The i-th coordinate of a generalized multivariate Hawkes process~$N$}
We start with
\begin{definition}\label{def:i-th}
We define the $i-th$ coordinate $N^i$ of $N$ as
\begin{align} \label{N-i}
N^i((0,t], A) := \sum_{n\geq 1} \1_{\{ T_n \leq t  \}}\1_{\{X_n \in A^i \}},
\end{align}
for $A \in \cE_i$ and $t\geq 0$, where
\begin{equation}\label{A-i}
A^i =
\Big(
\bigtimes_{j=1}^{i-1}   E_j^\Delta
\Big)
\times A \times
\Big(
\bigtimes_{j=i+1}^{d}   E_j^\Delta
\Big).
\end{equation} \qed
\end{definition}

Clearly, $N^i$ is  a MPP  and
\[N^i((0,t], A)=N((0,t],A^i).\]
Indeed, the $i$-th coordinate process $N^i$ can be represented as a sequence $N^i =(T^i_k, Y^i_k)_{k \geq 1}$, which is related to the sequence $(T_n, X^i_n)_{n \geq 1}$ as follows
\begin{align}\label{eq:Ni-def3}
(T^i_k, Y^i_k) &= \begin{cases}
(T_{m^i_k}, X^i_{m^i_k}) & \text{if  } m^i_k < \infty,  \\
(T_{m^i_{\hat k^i}+k-\hat k^i}, \Delta)  & \text{if  }  m^i_k = \infty \text{ and } T_{\infty}<\infty,\\
(\infty, \partial^i) & \text{if  }  m^i_{k}=\infty \text{ and } T_{\infty}=\infty,
\end{cases}		
\end{align}
where $\hat k^i= \max\{n\,:\, m^i_n < \infty\}$,
with $m^i$ defined as
\begin{align*}
m^i_1 &= \inf \set{ n \geq 1 : X^i_n \in E_i },  \\
m^i_k &= \inf \set{ n > m^i_{k-1} :  X^i_n \in E_i } \quad\textnormal{for } k >1.  	
\end{align*}

We clearly have

\begin{equation}\label{eq:NHi}
N^i((0,t], A) = \sum_{k\geq 1} \1_{\{ T^i_k \leq t  \}}\1_{\{Y^i_k \in A \}}, \quad  A \in \cE_i.
\end{equation}

In particular this means that for the $i$-th  coordinate $N^i$ the times $T_n(\omega)$ such that $X^i_n(\omega) = \Delta$ are disregarded as event times for this coordinate since the events occurring with regard to the entire $N$ at these times do not affect the $i$-th coordinate.

We define the completed filtration	 ${\mathbb F}^{N^i}=(\F^{N^i}_t,\ t\geq 0)$
generated by $N^i$ in analogy to ${\mathbb F}^N$; specifically  $\F^{N^i}_t$
is the $\bP$--completion of the $\sigma$--field
$
\sigma(N^i((s,r]\times A)\, :\, 0\leq s<r\leq t,\ A\in \cE_i),\ t\geq 0.
$
In view of Theorem 2.2.4 in \cite{LasBra1995} the filtration ${\mathbb F}^{N^i}$ satisfies the usual conditions.

We  define the explosion time $T^i_\infty$ of $N^i$ as
\[
	T^i_\infty := \lim_{n \rightarrow \infty} T^i_n.
	\]
Clearly, $T^i_\infty \leq T_\infty$.

We conclude this section with providing some more insight into the properties of the measure $N^i$. Towards this end, we first observe that the measure $N^i$ is both $\bF^N$--optional and $\bF^{N^i}$--optional. Subsequently, we will derive  the compensator of $N^i$ with respect to $\bF^N$ and  the compensator of $N^i$ with respect to $\bF^{N^i}.$ The following Proposition \ref{prop:compfulli} and Proposition \ref{prop:nutilde} come handy in this regard.

\begin{proposition}\label{prop:compfulli}
 Let $N$ be a generalized multivariate Hawkes process with Hawkes kernel $\kappa$.
 Then the $({\mathbb F}^N,\mathbb{P})$--compensator, say $\nu^i,$ of measure $N^i$ defined in \eqref{N-i} is given as
\begin{equation}\label{eq:nui}
\nu^i(\omega,dt,dy_i)=\1_{]\!] 0; T^i_\infty(\omega) [\![} (t)\kappa^i(\omega,t,dy_i)dt,
\end{equation}
where
\begin{equation}\label{eq:kappai}
\kappa^i(t,A):=\kappa(t,A^i),\quad t\geq 0,\ A\in \mathcal{E}_i,
\end{equation}
with $A^i$ defined in \eqref{A-i}.
\end{proposition}
\begin{proof}	
According to  Theorems 4.1.11 and 4.1.7 in \cite{LasBra1995} the $i$-th coordinate  $N^{i}$ admits a unique ${\mathbb F}^N$--compensator, say $\nu^i$, with a property that $\nu^i([\![T^i_\infty; \infty [\![ \times E_i) = 0$.
For every $n$ and $A\in \cE_i$ the processes $M^{i,n,A}$ and $\widehat M^{i,n,A}$ given as
\[
	M^{i,n,A}_t = N^i((0,t \wedge T_n ] \times  A ) - \int_0^{t \wedge T_n} \1_{]\!] 0;T_\infty [\![} (u)\kappa^i(u,A)du, \ t\geq 0,
\]
and
\[
	\widehat M^{i,n,A}_t = N^i((0,t \wedge T_n ] \times  A ) -
	\nu^i( (0,t \wedge T_n ] \times  A ) , \ t\geq 0,
\]
are $({\mathbb F}^N,\bP)$--martingales. Hence the process
\[
\int_0^{t \wedge T_n} \Big( \1_{]\!] 0; T_\infty [\![} (u)\kappa^i(u,A)du - \nu^i( du ,  A )\Big), \  t\geq 0,
\]
is an $\bF^N$-predictable martingale.
Since it is of integrable variation and null at $t=0$ it is null for all $t\geq 0$ (see e.g. Theorem VI.6.3 in \cite{HeWanYan1992}).
From the above and the fact that $T^i_\infty \leq T_\infty$  we deduce that
\begin{align*}
\int_0^{t \wedge T_n} \1_{]\!] 0 ; T^i_\infty [\![} (u)\kappa^i(u,A) du
&= \nu^i( (0,t \wedge T_n ] \times  A ) , \qquad t\geq 0.
\end{align*}
This proves the proposition.
\end{proof}

\begin{remark}\label{rem:krg-ker} Note that for each $i$, the function  $\kappa^i$ defined in \eqref{eq:kappai} is a measurable kernel from $(\Omega \times \mathbb{R}_+,  \cP  \otimes \mathcal{B}(\mathbb{R}_+))$ to $(E_i,\cE_i)$.  It is important to observe that, in general,  there is no one-to-one correspondence between the Hawkes kernel $\kappa$ and all the marginal kernels $\kappa^i$, $i=1,\ldots,d$. We mean by this  that may exist another Hawkes kernel, say $\widehat \kappa$, such that $\widehat \kappa \neq \kappa$ and
	\begin{equation}\label{eq:kappai-hat}
	\kappa^i(t,A)=\widehat \kappa(t,A^i),\quad t\geq 0,\ A\in \mathcal{E}_i,\ i=1,\ldots d.
	\end{equation}	
	\qed		
\end{remark}

\begin{remark}\label{rem:law-of-N}
		As we know from  Remark \ref{rem:law-of-N-under-P} the Hawkes kernel $\kappa$ determines the law of $N$.
	However, in view of Remark \ref{rem:krg-ker}, the kernel $\kappa^i$ may not determine the law of $N^i$. It remains to be an open problem for now to determine  sufficient conditions under which the law of $N^i$ is determined by $\kappa^i$. This problem is a special case of a more general problem: what are  general sufficient conditions under which characteristics of a semimartingale determine the law of this semimartingale.
	\qed
\end{remark}
The following important result gives the ${\mathbb F}^{N^i}$--compensator of measure $N^i$.
\begin{proposition}\label{prop:nutilde}
	Let $N$ be a generalized multivariate Hawkes process with Hawkes kernel $\kappa$.
Then the ${\mathbb F}^{N^i}$--compensator of measure $N^i$, say $\widetilde \nu^i,$ is given as
\begin{equation}\label{eq:tildenui}
\widetilde \nu^i(\omega,dt,dy_i)=(\nu^i)^{p, {\mathbb F}^{N^i}}(\omega,dt,dy_i),
\end{equation}
where $(\nu^i)^{p, {\mathbb F}^{N^i}}$
is the dual predictable projection of $\nu^i$ on ${\mathbb F}^{N^i}$ under $\bP$.
\end{proposition}
\begin{proof}
Using Theorems 4.1.9 and 3.4.6 in \cite{LasBra1995}, as well as the uniqueness of the compensator, it is enough to show that for any $A\in \cE^i$ and any $n\geq 1$ the process $(\nu^i)^{p, {\mathbb F}^{N^i}}((0,t\wedge T^i_n],A)$,
where $(\nu^i)^{p, {\mathbb F}^{N^i}}$ is the dual predictable projection of $\nu^i$ on ${\mathbb F}^{N^i}$ under $\bP$, is
the ${\mathbb F}^{N^i}$--compensator of the increasing process $N^i((0,t\wedge T^i_n], A),\ t\geq 0$. This however follows from Theorem 3.3 in \cite{BieJakJeaNie2018}.
\end{proof}

\subsection{Examples}

We will provide now some examples of generalized multivariate Hawkes processes.

For $\omega = (t_n,x_n)_{n\geq 1}$,  $t \geq 0$ and  $A \in \cE^\Delta$ we set
\begin{equation}\label{eq:NN}
	N(\omega, (0,t] ,  A) := \sum_{ n \geq 1} \1_\set{ t_n \leq t, x_n \in A}.
\end{equation}

In all examples below we define the kernel $\kappa$ of the form \eqref{eq:kappa-G} with $\eta$ and $f$ properly chosen, so that we may apply Theorem \ref{prop:main} to the effect that there exists a probability measure $\bP_\nu$ on $(\Omega, \mathcal{F})$ such that process  $N$ given by \eqref{eq:NN} is a Hawkes process with the Hawkes kernel equal to $\kappa$. In other words, there exists a probability measure $\bP_\nu$  on $(\Omega, \mathcal{F})$ such that $\nu$ given in \eqref{eq:nu-Hawkes-G}--\eqref{eq:kappa-G} is the $\bF^N$--compensator of $N$ under $\bP_\nu$.

 For a Hawkes process $N$ with a mark space $E^\Delta$ we introduce the following notation
\[
N_t = N((0,t], E^\Delta), \quad t\geq 0.
\]
	Likewise, we denote for $i=1,\ldots,d,$
\[N^i_t=N^i((0,t], E_i), \quad t\geq 0.\]

\begin{example}\label{ex:clasicHawkes}
	\textbf{Classical univariate Hawkes process}

We take $d=1$ and $E_1=\{1\}$, so that $E^\Delta = E_1=\{1\}$.  As usual, and in accordance with \eqref{eq:NH-G}, we identify  $N$ with a point process $(N_t)_{t \geq 0}$. 
Now we take
\[\eta(t,\{1\})=\lambda(t),\]
where $\lambda$ is positive, locally integrable function, and, for $0\leq s\leq t$, we take
\[ f(t,s,1,\{1\})=w(t-s)\]
for some non-negative function $w$ defined on $\R_+$ (recall that $ f(t,s,1,\{1\})=0$ for $ s\geq t$).
Using these objects we define $\kappa$ by
\[
\kappa(t,dy) = \bar{\kappa}(t)\delta_{\{1\}}(dy),
\]
where
\[
\bar{\kappa}(t)=\lambda(t)+\int_{(0 ,t)}\, w(t-s)dN_s.
\]

In case of the classical univariate Hawkes process sufficient conditions under which the explosion time is almost surely infinite, that is
\[T_\infty =\infty \quad \bP_\nu - \textrm{a.s.}\]
are available in terms of the Hawkes kernel.
Specifically,  sufficient conditions for no-explosion are given in 
\cite{BDHM}:
\[
\lambda \text{ is locally bounded}, \quad \text{and} \quad \int_0^\infty w(u) du < \infty. 
\qquad \qed
\]
\end{example}

\begin{example} \textbf{Generalized bivariate Hawkes process with common event times}\label{commonevents}

In the case of a generalized  bivariate Hawkes process $N$ we have $d=2$ and the mark space is given as
 \begin{equation*}
 	E^\Delta = E^\Delta_1 \times E^\Delta_2 \setminus \{ (\Delta, \Delta)\}
 	=
 	\{
 	(\Delta, y_2), (y_1, \Delta), (y_1, y_2):  y_1 \in E_1, y_2 \in E_2
 	\}.
\end{equation*}
Here, in order to define kernel $\kappa$, we take kernel $\eta$ in the form
\[
\eta(t,dy)=
\eta_{1}(t,d y_1) \otimes\delta_{\Delta} (d y_2)
+
 \delta_{\Delta} (d y_1) \otimes \eta_{2}(t,d y_2)
+
\eta_c (t,d y_1, d y_2 ),
\]
where
$\delta_\Delta$ is a Dirac measure, $\eta_{i}$  for $i=1,2$ are probability kernels, from $(\R_+, \cB(\R_+))$ to $(E_i, \cE_i)$
and $\eta_{c}$ is a
probability kernel from $(\R_+, \cB(\R_+))$ to $(E^\Delta, \cE^\Delta)$, satisfying
\[
\eta_c (t,E_1 \times \Delta)
=
\eta_c (t,\Delta \times E_2 ) =0.
\]
Kernel $f$ is given, for $0\leq s\leq t$ and $x=(x_1,x_2)$, by
\begin{align}\nonumber
	f(t,s,& \, x ,dy )
	=    \Big(  w_{1,1} (t,s)g_{1,1}(x_1)\1_{E_1 \times\Delta } (x)  + w_{1,2} (t,s)g_{1,2}(x_2) \1_{\Delta \times E_2} (x) \\ \nonumber
	& + w_{1,c} (t,s) g_{1,c}(x) \1_{E_1 \times E_2 } (x) \Big) \phi_1(x, dy_1 ) \otimes\delta_\Delta(dy_2)  \\ \nonumber
	&+  \Big(  w_{2,1}(t,s)g_{2,1}(x_1)  \1_{E_1 \times\Delta } (x)  +  w_{2,2}(t,s)g_{2,2}(x_2) \1_{\Delta \times E_2} (x) \\ \label{eq:f-Hawkes}
	& +  w_{2,c}(t,s) g_{2,c}(x) \1_{E_1 \times E_2 } (x)\Big) \delta_\Delta(dy_1)  \otimes \phi_2(x, dy_2) \\ \nonumber
	&+
		\Big(  w_{c,1}(t,s) g_{c,1}(x_1)  \1_{E_1 \times\Delta } (x) +  w_{c,2}(t,s) g_{c,2}(x_2) \1_{\Delta \times E_2} (x)    \\ \nonumber
		& + 	 w_{c,c}(t,s) g_{c,c}(x_1, x_2) \1_{E_1 \times E_2 } (x)  \Big) \phi_c (x, dy_1, dy_2),
\end{align}
where $\phi_i$ is a probability kernel from $(E^\Delta, \cE^\Delta )$ to $(E_i, \cE_i)$ for $i=1,2$ and
$\phi_c$ is a probability kernel from $(E^\Delta, \cE^\Delta )$ to $(E^\Delta, \cE^\Delta )$ satisfying
\[
\phi_c (x, E_1 \times \Delta)
=
\phi_c (x, \Delta \times E_2)=0.
\]
The \textit{decay functions} $w_{i,j}$
 and the \textit{impact functions} $g_{i,j}$,
 $i,j=1,2,c$, are appropriately regular and deterministic. Moreover, the decay functions are positive and the impact functions are non-negative.  In particular, this implies that  the kernel $f$ is  deterministic and non-negative.

In what follows we will need  the concept of idiosyncratic group of $I$ coordinates of a generalized bivariate Hawkes process $N$.
For $I=\set{1}$ we define
	\[
	N^{idio,\set{1}}((0,t],A):=N((0,t],A \times \Delta),  \quad t \geq 0, \ A \in \cE_1
	\]
and, likewise, for $I=\set{2}$ we define
\[
N^{idio,\set{2}}((0,t],A):=N((0,t],\Delta \times A),   \quad t \geq 0, \ A \in \cE_2.
\]
Finally, for $I=\set{1,2}$ we define
\[
N^{idio,\set{1,2}}((0,t],A):=N((0,t],A),   \quad t \geq 0, \ A \in \cE_1 \otimes \cE_2.
\]

\medskip \noindent
Clearly, $N^{idio,\mathcal{I}}$ is a MPP. For example, $N^{idio,i}$ is a MPP which  records idiosyncratic events occurring with regard to $X^i$;
that is, events that only regard to $X^i$, so that $X^j_{n}=\Delta$ for $j\ne i$  at times $T_n$ at which these events take place.
Likewise,  $N^{idio,\set{1,2}}$ is a MPP which  records idiosyncratic events occurring with regard to  $X^1$ and $X^2$ simultaneously.
Let us note that \[
N^1 = N^{idio, \set{1}} + N^{idio, \set{1,2}}, \qquad
N^2 = N^{idio, \set{2}} + N^{idio, \set{1,2}}.
\]

\medskip

 We will now interpret various terms that appear in the expressions for $\eta$ and $f$ above:
\begin{itemize}

        \item $\eta_{1}(t,d y_1) \otimes\delta_{\Delta} (d y_2)$ represents autonomous portion of the intensity, at time $t$, of marks of the coordinate $N^1$  taking values in the set $dy_1\subset E_1$ and no marks occurring for $N^2$;

        \item $\eta_c (t,d y_1, d y_2 )$  represents autonomous portion of the intensity, at time $t$, of an event amounting to the marks of both coordinates $N^1$ and $N^2$ taking values in the set $dy_1dy_2\subset E_1\times E_2$;

   \item \begin{align*} &\int_{(0,t)\times E^\Delta }w_{1,1} (t,s)g_{1,1}(x_1)\1_{E_1 \times\Delta } (x)\phi_1(x, dy_1 )
   \otimes\delta_\Delta(dy_2)N(ds,dx) \\ = &\int_{(0,t)\times E_1 }w_{1,1} (t,s)g_{1,1}(x_1) \phi_1((x_1,\Delta), dy_1 )
   \otimes\delta_\Delta(dy_2)N^{idio,1}(ds,dx_1)
   \end{align*}
    represents idiosyncratic impact of the coordinate $N^1$ alone on the intensity, at time $t$, of marks of the coordinate $N^1$  taking values in the set $dy_1\subset E_1$ and no marks occurring for $N^2$;

   \item \begin{align*}
    &\int_{(0,t)\times E_2 }w_{1,2} (t,s)g_{1,2}(x_2)\phi_1((\Delta,x_2), dy_1 )
   \otimes\delta_\Delta(dy_2)N^{idio,2}(ds,dx_2)
   \end{align*}      represents idiosyncratic  impact of the coordinate $N^2$ alone on the intensity, at time $t$, of an event amounting to the marks of coordinate $N^1$ taking value in the set $dy_1\subset E_1$ and no marks occurring for $N^2$;

    \item \[ \int_{(0,t)\times E^\Delta }w_{1,c} (t,s)g_{1,c}(x) \1_{E_1 \times E_2 } (x)\phi_1(x, dy_1 ) \otimes\delta_\Delta(dy_2)N(ds,dx)\] represents joint impact of the coordinates $N^1$ and $N^2$ on the intensity, at time $t$, of an event amounting to the marks of coordinate $N^1$ taking value in the set $dy_1\subset E_1$ and no marks occurring for $N^2$;

        \item \begin{align*}
        &\int_{(0,t)\times E_1 }w_{c,1} (t,s)g_{c,1}(x_1)\phi_c((x_1,\Delta), dy_1,dy_2 )N^{idio,1}(ds,dx_1)
   \end{align*}
   represents idiosyncratic  impact of the coordinate $N^1$ alone on the intensity, at time $t$, of an event amounting to the marks of both coordinates $N^1$ and $N^2$ taking values in the set $dy_1dy_2\subset E_1\times E_2$;

   \item \[\int_{(0,t)\times E^\Delta }w_{c,c}(t,s)g_{c,c}(x_1)  \1_{E_1 \times\Delta } (x)\phi_c (x, dy_1, dy_2)N(ds,dx)\] represents joint  impact of the coordinates $N^1$ and $N^2$ on the intensity, at time $t$, of an event amounting to the marks of both coordinates $N^1$ and $N^2$ taking values in the set $dy_1dy_2\subset E_1\times E_2$.
                  \end{itemize}

In particular, the terms contributing to occurrence of common events are $\eta_c (t,d y_1, d y_2 )$ and \[\Big(  g_{c,1}(x_1)  \1_{E_1 \times\Delta } (x) +  g_{c,2}(x_2) \1_{\Delta \times E_2} (x)    + 	 g_{c,c}(x_1, x_2) \1_{E_1 \times E_2 } (x)  \Big) \phi_c (x, dy_1, dy_2).\]

Upon integrating $\kappa(t,dy)$ over  $A_1 \times \{ \Delta, E_2 \}$ we get
\allowdisplaybreaks
  \begin{align*}
 \kappa^1&(t,A_1)
 =\kappa(t, A_1 \times 	\{ \Delta, E_2 \} ) \nonumber \\
 & = \eta_1(t,A_1) + \eta_c(t,A_1 \times E_2)	\nonumber \\
 &+  \int_{(0,t)\times E_1}
 \Big( w_{1,1} (t,s)  g_{1,1}(x_1)  \phi_1((x_1, \Delta), A_1)  \nonumber \\
 &  \qquad\qquad \qquad+ w_{c,1}(t,s)
 g_{c,1}(x_1)  \phi_c ((x_1, \Delta), A_1 \times E_2 )\Big)  N^{idio,1}(ds,dx_1) \nonumber  \\ 	
 &+  \int_{(0,t)\times E_2} \Big(w_{1,2} (t,s)  g_{1,2}(x_2)   \phi_1((\Delta,x_2), A_1) \\ \nonumber
 &  \qquad\qquad + w_{c,2}(t,s)
 g_{c,2}(x_2) \phi_c ((\Delta,x_2), A_1 \times E_2 )\Big) N^{idio,2}(ds,dx_2)  \\ \nonumber		
 &+  \int_{(0,t)\times E}\Big( w_{1,c} (t,s)   g_{1,c}(x) \phi_1((x_1,x_2), A_1)  \nonumber \\
 &\qquad +    w_{c,c}(t,s) g_{c,c}(x_1, x_2) \phi_c ((x_1,x_2), A_1 \times E_2 ) \nonumber
 \Big)
 N^{idio, \set{1,2}}(ds,dx).
 \nonumber
 \end{align*}

To complete this example we note that upon setting $\eta_c=0$ and $\phi_c=0$ we produce a generalized bivariate Hawkes process with no common event times.\qed
\end{example}

\section{Mathematical construction of and computational pseudo-algorithm for simulation of a generalized multivariate Hawkes process with deterministic kernels $\eta$ and $f$}\label{constr}

Fix an arbitrary $T >0$.
In this section we first provide a construction of restriction to $[0,T] \times E^\Delta$
 of a generalized multivariate Hawkes process, with deterministic kernels $\eta$ and $f$, via Poisson thinning, that is motivated by a similar construction given in \cite{BJN-buczek}. Then, based on our construction, we present a computational pseudo-algorithm for simulation of a generalized multivariate Hawkes process restricted to $[0,T] \times E^\Delta$.

We are concerned here with a generalized multivariate Hawkes process
 admitting the  Hawkes kernel  of the form
\begin{equation}\label{new-nu}
\kappa(t,dy)= {\eta}(t,dy)+\int_{(0, t)\times E^\Delta } {f}(t,s,x,dy) N(ds,dx),
\end{equation}
where  ${\eta}$ is a deterministic finite kernel from $(\bR_+, \cB(\bR_+))$ to $(E^\Delta, \mathcal{E}^\Delta)$ and ${f}$ is a deterministic finite kernel from $(\bR^2_+ \times E^\Delta, \cB(\bR^2_+) \otimes \cE^\Delta)$ to $(E^\Delta, \mathcal{E}^\Delta)$.

We may, and we do, represent kernels $\eta$, $f$ as
\[
\eta(t,dy) = \eta(t,E^\Delta)Q_1(t,dy) , \quad \text{ where } \quad Q_1(t,dy) =
\frac{\eta(t,dy) }{\eta(t,E^\Delta) }
\1_\set{\eta(t,E^\Delta)  > 0  } +
\delta_\partial (dy) \1_\set{\eta(t,E^\Delta)  = 0  }
, 
\]
and
\[
f(t,s,x,d y) = f(t,s,x,E^\Delta) Q_2(t,s,x,dy) ,
\]
where
\[
Q_2(t,s,x,dy) =  \frac{f(t,s,x,d y)  }{f(t,s,x,E^\Delta) }
\1_\set{f(t,s,x,E^\Delta) > 0  }
+
\delta_\partial (dy) \1_\set{f(t,s,x,E^\Delta) = 0  }
.
\]
Note that
$Q_1$ and $Q_2$ are deterministic probability kernels.

Since we are concerned with a restricted Hawkes process we consider a Hawkes kernel $\kappa_T$ which is a restriction to $[0,T]$ of $\kappa$ that is
\begin{equation}\label{eq:kapparest}
\kappa_T(t,dx) = \1_{[0,T]} (t)\kappa(t,dx).
\end{equation}

For simplicity of notation we suppress $T$ in the notation below. So, for example, we will write $f$ rather than $f_T:=\1_{[0,T]}f$.

We make the following standing assumption:

\[
\sup_{t \in [0,T]} \eta(t,E^\Delta) \leq \wh \eta <\infty,\]
for some constant $\wh{\eta} > 0 $
and, for $s\in [0,T]$ and $x\in E^\Delta$
\begin{equation}\label{eq:hatf}
 \sup_{ t \in {[s,T]}} f(t,s,x,E^\Delta)  \leq \wh{f}(s,x) < \infty,
\end{equation}
for some measurable mapping $ \wh{f}: [0,T] \times E^\Delta \rightarrow (0,\infty)$.

{
\subsection{Description of the construction}\label{sec:constr}

Now we  describe a construction of Hawkes process with Hawkes kernel given by \eqref{eq:kapparest}. This construction leads immediately to a pseudo-algorithm, presented in the next section, for simulation of such Hawkes process.

In what follows we will define recursively a sequence of random measures $(N^k)_{k \geq 0}$ that  provide building blocks for our Hawkes process.

Towards this end we first let $\beta$ be the Borel isomorphism between the space $E^\partial$ and a Borel subset of $\bR^d\cup \widehat \partial$, with the convention that $\beta (\partial) = \widehat \partial$.

Our construction will proceed in several steps.

\medskip

Step 1). Let us consider an array $
\{
(Z^{k,j},(U^{k,j}_i, V^{k,j}_i, W^{k,j}_i)_{i=1}^\infty)\}_{k=0,j=1}^\infty$ of independent identically distributed random variables with uniform distribution on $(0,1]$.
Let $D:[0, \infty) \times (0,1] \rightarrow \bN$ be a measurable function such that
\[
\int_{(0,1]} \1_\set{ k }(D(\lambda,u) ) du  = e^{-\lambda}\frac{\lambda^k}{k!},\quad  k=0,1, \ldots
\]
where we use the convention that $0^0=1$. Therefore, for a random variable $U$ uniformly distributed on $(0,1]$ the random variable $D(\lambda,U)$ has Poisson distribution with parameter $\lambda \geq 0$, where we extend the concept  of Poisson distribution by allowing $\lambda = 0$.
Moreover let $G_1: [0,T] \times (0,1] \rightarrow E^\Delta$ be a measurable function  such that
\[
\int_{(0,1]}  \1_{A}(G_1(t, u)) du
=Q_1(t,A), \quad A \in \cE^\Delta,
\]
and
 $G_2: [0,T] \times [0,T] \times E^\Delta \times (0,1] \rightarrow E^\Delta$  be a measurable function  such that
\[
\int_{(0,1]} \1_{A }(G_2(t, s, y, u)) du
=Q_2(t,s,y,A), \quad A \in \cE^\Delta.
\]
Existence of such functions $G_1$ and $G_2$ is asserted by Lemma 3.22 in \cite{Kal2002}.

 We use the left open intervals of integration above so to be consistent with the the rest of the construction. The reason that we work with left open intervals in the rest of the construction  is that the births of the offsprings occur after the appearance of their parents (e.g., after arrivals of the immigrants), see  Section \ref{sec:cluster}. This feature is explicitly accounted for in the construction presented here.

\medskip

Step 2). Using $(Z^{0,1},(U^{0,1}_i, V^{0,1}_i, W^{0,1}_i)_{i=1}^\infty)$  we define a random measure $N^0$ on $\mathcal{B}(\bR_+)\otimes \cE^\Delta$:
\begin{equation}\lab{N0}
N^0(dt,dx)
= \sum_{i=1}^{\infty} \delta_{(T^0_i, X^0_i)}(dt,dx) \1_\set{ i \leq P^0, A^0_i \leq  \eta(T^0_i, E^\Delta) },
\end{equation}
where
\[
P^0 = D(T\hat{\eta}, Z^{0,1}),\quad
T^{0}_i = TU^{0,1}_i, \quad A^{0}_i =\hat{\eta} V^{0,1}_i , \quad X^{0}_i = G_1(T^{0}_i, W^{0,1}_i) .
\]

We note that $P^0 $ is a Poisson random variable with parameter $T \wh \eta $,  which is independent of the
iid sequence $(T^0_i, X^0_i, A^0_i)_{i=1}^\infty$
of random elements with values in
$ (0,T] \times [0,\wh \eta] \times E^\Delta$, and that
\[
\bP((T^0_i,A^0_i, X^0_i)\in dt\times da \times dx)= \frac{1}{T \wh \eta} \1_{(0,T]\times [0,\wh \eta]}(t,a) Q_1(t,dx) dt da.
\]

Then, we consider a sequence $(S^0_j, Y^0_j)_{j=1}^\infty$ with $S^0_j\in (0,T]\cup \{\infty\}$ given as
\begin{align*}
S^0_j &:= \inf \set{ t\geq 0: N^0((0,t] \times E^\Delta)  \geq j}=\inf \set{ t\in [0,T]: N^0((0,t] \times E^\Delta)  \geq j},
\end{align*}
and with $Y^0_j$ constructed as follows:
\[
Y^0_j(\omega) :=\left\{
\begin{array}{ll}
X^0_i(\omega), & \hbox{for $i$ such that  $S^0_j(\omega)=T^0_i(\omega)<\infty$;} \\
\partial, & \hbox{if $S^0_j(\omega)=\infty$.}
\end{array}
\right.\]
The sequence $(S^0_j, Y^0_j)_{j=1}^\infty$ is well defined
because $N^0$ is a counting measure such that  $N^0(\set{t} \times E^\Delta)\leq 1$ for $t\geq 0$, so $N^0(\set{S^0_j} \times E^\Delta)= 1$,
provided ${S^0_j < \infty }$, and since $\bP( \exists i\neq k \ : \ T^{{0}}_i = T^{{0}}_k)=0$.
Moreover, $Y^0_j$ is	 a random element. Indeed, 
	\begin{align*}
	 Y^0_j &= \beta^{-1}\Big(\sum_{i=1}^\infty \beta ( X^0_i)\1_\set{T^0_i=S^0_j}\1_\set{S^0_j<\infty} + \widehat \partial \1_\set{S^0_j=\infty}\Big)
	 \\
	 &=
	\beta^{-1}\Big(
	\1_\set{ S^0_j < \infty }
	\int_{E^\Delta} \beta(x) N^0(\set{S^0_j} \times dx)
	+\widehat \partial \1_\set{ S^0_j = \infty }\Big).
	\end{align*}

Observe that $S^0_j < S^0_{j+1}$ on $\set{S^{0}_j < \infty}$, and that the measure $N^0$ may be identified with the sequence $(S^0_j, Y^0_j)_{j=1}^\infty$. {Indeed,
defining   $\Psi^0:=\textrm{card} \{j\geq 1\, :\, S^0_j<\infty\}$, we have $\Psi^0\leq P^0$, and thus $\bP(\Psi^0<\infty)=1$. Consequently,
\begin{equation}\lab{N0-1}
N^0(dt,dx)
= \sum_{j = 1}^\infty \delta_{(S^0_j, Y^0_j)}(dt,dx) \1_\set{S^0_j <\infty}=\sum_{j = 1}^{\infty} \delta_{(S^0_j, Y^0_j)}(dt,dx)\1_\set{j \leq \Psi^0}.
\end{equation}
The representation \eqref{N0-1} is more convenient for our needs than the representation \eqref{N0}. This is because the sequence $(S^0_j, Y^0_j)_{j=1}^\infty$ is ordered with respect to the first component, so that this sequence is a MPP  and thus measure $N^0$ may also be considered as a MPP.}

\medskip

  Step 3). Now, we proceed by recurrence. So,  for $k \in \bN$ suppose that we have constructed a random sequence $(S^k_j, Y^k_j)_{j=1}^\infty$  with the property that
$S^{k}_j \in (0,T] $ if $\set{S^{k}_j < \infty }$, and $\bP(\Psi^k < \infty) = 1$ where  $\Psi^k=\textrm{card} \{j\geq 1\, :\, S^k_j<\infty\}$, and that we have also constructed a random measure $N^k$ on $\mathcal{B}(\bR_+)\otimes \cE^\Delta$ satisfying
\[
N^k(dt,dx) = \sum_{j=1}^\infty \delta_{(S^k_j, Y^k_j)}(dt,dx) \1_\set{ S^k_j < \infty}= \sum_{j=1}^{\infty} \delta_{(S^k_j, Y^k_j)}(dt,dx)\1_\set{j \leq \Psi^k }.
\]
Given
$N^k$, or equivalently given $(S^k_j, Y^k_j)_{j=1}^{\infty}$,  we will define a sequence of random measures
$(N^{k+1}_j)_{j \geq 1} $, which are conditionally independent given {$\sigma(N^0, \ldots, N^k)$}.\footnote{Conditional independence between random measures is understood as conditional independence between random elements taking values in the space of probability measures. We refer to Kallenberg \cite{Kal2002}, Chapter 12, for definition of random elements taking values in the space of $\sigma$-finite measures on a measurable space, and to Chapter 6 therein for definition of conditional independence between random elements. }
Fix $j\in \set{1,2,\ldots}$. We let {the} random measure $N^{k+1}_j$ on $\mathcal{B}(\bR_+)\otimes \cE^\Delta$ be
defined by
\begin{equation}\label{eq:Nkp1j}
N^{k+1}_j (dt,dx) = \sum_{i =1}^{\infty } \delta_{(T^{k+1,j}_i, X^{k+1,j}_i)} (dt,dx) \1_\set{S^k_j < T, \, i \leq P^{k+1}_j, \, A^{k+1,j}_i \leq f(T^{k+1,j}_i, S^k_j, Y^k_j ,E^\Delta) },
\end{equation}
where $P^{k+1}_j$, $(T^{k+1,j}_i,A^{k+1,j}_i, X^{k+1,j}_i)_{i=1}^{\infty}$
are  random variables defined by transformation of the sequence
$Z^{k+1,j}$,$(U^{k+1,j}_i, V^{k+1,j}_i, W^{k+1,j}_i)_{i=1}^\infty$ and the pair $(S^k_j, Y^k_j)$  in the following way:
{
\begin{align}\label{eq:Pkp1j}
P^{k+1}_{j} &= D\big ((T - S^k_j) \wh f(S^k_j,Y^k_j) \1_\set{ S^k_j <T},Z^{k+1,j}\big )
\\
\nonumber
& = D\big ((T - S^k_j) \wh f(S^k_j,Y^k_j) ,Z^{k+1,j}\big ) \1_\set{ S^k_j <T}
\\
\nonumber
T^{k+1,j}_i &= (S^k_{j} + (T-S^k_j)U^{k+1,j}_i )\1_\set{ S^k_j <T} + \infty \1_\set{ S^k_j \geq T },
\\
\nonumber
 A^{k+1,j}_i &=\hat{f}(S^k_j, Y^k_j) V^{k+1,j}_i \1_\set{ S^k_j <T}, 
\\
\nonumber
X^{k+1,j}_i &= G_2(T^{k+1}_i, S^{k}_j, Y^{k}_j, W^{k+1,j}_i) \1_\set{ S^k_j <T} + \partial \1_\set{ S^k_j \geq T}.
\end{align}}}
Note that the random variable $P^{k+1}_j$ has
{$\sigma(N^0, \ldots , N^k)$-conditionally}
Poisson distribution with parameter $(T - S^k_j) \wh f(S^k_j,Y^k_j) \1_\set{ S^k_{j} <
	\infty } $, where $\wh f$ given by \eqref{eq:hatf}.
The  random elements in the sequence
$(T^{k+1,j}_i,A^{k+1,j}, X^{k+1,j}_i)_{i=1}^{\infty}$ take values in {$ (S^k_j,T] \times [0,\wh f(S^k_j,Y^k_j)] \times E^\Delta $ 	if $S^k_{j} < T$; otherwise, if $S^k_j \geq T$, these elements are all constant and equal to $(\infty,0,\partial)$.}
Moreover, they are $\sigma(N^0, \ldots, N^k)$-conditionally independent random elements, and
the   $\sigma(N^0, \ldots, N^k)$-conditional distribution of $(T^{k+1,j}_i,A^{k+1,j}_i, X^{k+1,j}_i)$ is given by
\begin{align}
\nonumber
&\bP((T^{k+1,j}_i,A^{k+1,j}, X^{k+1,j}_i)\in dt\times da \times dx | N^0, \ldots, N^k)\\
\label{eq:distkp1j}
&=	
\1_\set{ S^k_j < T}\frac{1}{(T - S^k_j)\wh f(S^k_j,Y^k_j) }\1_{(S^k_j,T]\times [0,\wh f(S^k_j,Y^k_j)]}(t,a) Q_2(t,S^k_j,Y^k_j,dx) dt da
\\ \nonumber
& \quad +\1_\set{ S^k_j \geq T } \delta_{(\infty,0,\partial)}(dt,da,dx).
\end{align}
Thus  if $S^k_j \geq T $, then $N^{k+1}_{j}  \equiv 0$.
The random measure $N^{k+1}_j$ can be identified with the random sequence $(S^{k+1,j}_n, Y^{k+1,j}_n)^\infty_{n=1}$, where
\begin{align*}
S^{k+1,j}_n &:= \inf \set{ t: N^{k+1}_j((0,t] \times E^\Delta)  \geq n}
\\
Y^{k+1,j}_n &:=
{
\beta^{-1}
\Big(
\1_\set{ S^{k+1,j}_n < T }
\int_{E^\Delta} \beta(x) N^{k+1}_{j}(\set{S^{k+1,j}_n} \times dx)
+ \widehat\partial \1_\set{ S^{k+1,j}_n \geq T } \Big).}
\end{align*}
Indeed, we have
\[
N^{k+1}_j (dt,dx) = \sum_{i =1}^{\infty } \delta_{(S^{k+1,j}_i, Y^{k+1,j}_i)} (dt,dx) \1_\set{ S^{k+1,j}_i < \infty}
= \sum_{i =1}^\infty \delta_{(S^{k+1,j}_i, Y^{k+1,j}_i)} (dt,dx) \1_\set{ i \leq {\Psi^{k+1}_j } },
\]
where $\Psi^{k+1}_j = \textrm{card} \set{ i : S^{k+1,j}_i < \infty} $ is such that $\Psi^{k+1,j}_i \leq P^{k+1}_j <  \infty $ with probability 1.
Therefore, {since the sequence $(S^{k+1,j}_n)_{n=1}^\infty$ is increasing as long as $S^{k+1,j}_n<\infty$, the measure} $N^{k+1}_j$ is a  MPP.
{Moreover,  if $S^k_j < T$, then $T^{k+1,j}_{l}  \in (S^k_j, T]$ for every $l$. Hence using the definition of $S^{k+1,j}_i$ we conclude  that $S^{k+1,j}_i \in (S^{k}_j,T]$, when $S^{k+1,j}_i$ is finite.}

Next,  we define the random measure  $N^{k+1}$ on $\mathcal{B}(\bR_+)\otimes \cE^\Delta$:
\begin{equation}\label{eq:Nkp1}
N^{k+1}(dt,dx) := \sum_{j \geq 1} N^{k+1}_j(dt,dx) =\sum_{j \geq 1} N^{k+1}_j(dt,dx)	\1_\set{{S^{k}_j < T }},
\end{equation}
Similarly as above we observe that the random measure  $N^{k+1}$ can be identified with the random sequence
	$(S^{k+1}_n, Y^{k+1}_n)_{n=1}^\infty$, where
	\begin{align*}
	S^{k+1}_n &:= \inf \set{ t: N^{k+1}((0,t] \times E^\Delta)  \geq n},
	\\
	Y^{k+1}_n &:=
	{\beta^{-1}\Big(
	\1_\set{ S^{k+1}_n < \infty }
	\int_{E^\Delta} \beta(x) N^{k+1}(\set{S^{k+1}_n} \times dx)
	+ \widehat \partial \1_\set{ S^{k+1}_n = \infty } \Big).}
	\end{align*}
	{Indeed, we have
\[
	\Psi^{k+1} := \textrm{card} \set{ i : S^{k+1}_i < \infty }
	  = \sum_{j=1}^{\Psi^k} \Psi^{k+1}_j,
	\]
and so $\bP(\Psi^{k+1} < \infty) = 1$. {
Moreover, we observe that  $N^{k+1}(\set{S^{k+1}_n} \times E^\Delta) = 1$, provided that $ S^{k+1}_n < \infty$ and hence
$Y^{k+1}_n(\omega) = Y^{k+1,j}_i(\omega)$ for $i$ such that  $S^{k+1}_n(\omega)=S^{k+1,j}_i(\omega)<\infty$.}
Thus,
	\[
	N^{k+1}(dt,dx)
	=
	\sum_{i =1}^{\infty } \delta_{(S^{k+1}_i, Y^{k+1}_i)} (dt,dx) \1_\set{ S^{k+1}_i < \infty}
	= \sum_{i =1}^{\infty } \delta_{(S^{k+1}_i, Y^{k+1}_i)} (dt,dx)\1_\set{i \leq \Psi^{k+1}}.
	\]
Since the sequence $\left(S^{k+1}_n, Y^{k+1}_n\right)_{n=1}^\infty$ forms a MPP, then $N^{k+1}$ may be considered as a MPP.}

{Recall that if $S^k_j < T$, then $S^{k+1,j}_{l}  \in (S^k_j, T]$ for every $l$, which implies that $S^{k+1}_i \in [0,T]$ if $\set{S^{k+1}_i < \infty}$.}

\medskip
	Step 4).
	Define a sequence of random measures  $H^{k}$, $k \geq 1$, on $\mathcal{B}(\bR_+)\otimes \cE^\Delta$ in terms of the previously constructed marked point processes $(N^j)_{j \geq 0} $ by
\begin{equation}\label{eq:Hk}
		H^{k}(dt,dx) = \sum_{m=0}^{k} N^{m}(dt,dx) ,\quad k \geq 1 .
\end{equation}	
A marked point process, say $(S_n^{H^k},Y_n^{H^k})_{n=1}^\infty$, can be associated with $H^k$ in a way analogous to how the sequence $\left(S^{k+1}_n, Y^{k+1}_n\right)_{n=1}^\infty$ has been associated with $N^{k+1}$. Consequently, ${H^k}$ may be considered as an MPP.

\medskip
	Step 5). Repeat Step $3$ and Step $4$ infinitely many times to obtain
		limiting random measure $H^{\infty}$ on $\mathcal{B}(\bR_+)\otimes \cE^\Delta$ given by
		\begin{equation}\label{Hinf}
H^\infty(dt,dx) = \sum_{m=0}^\infty N^m(dt,dx).
		\end{equation}

\begin{remark}
It is important to note that all random measures introduced in the construction above do not charge any set $F \in \mathcal{B}(\bR_+)\otimes \cE^\Delta$ such that $F \subset  (T,\infty]\times E^\Delta$. So, for example, for any such set we have $H^\infty(F)=0$.
\end{remark}

\subsection{Justification of the construction}
	
	Now we will justify that the construction given in Steps 1--5 above delivers a generalized multivariate Hawkes process with the Hawkes kernel given in \eqref{new-nu}.
	Towards this end let us introduce the following filtrations:
	\begin{align*}
	\bH^k &=\set{\cH^k_t}_{t \in [0,\infty)}
	,
	&&
	\text{ where } &   &\cH^k_t := \cF^{N^0}_t \vee \ldots \vee \cF^{N^k}_t ,
	\\
	\bH^\infty &= \set{\cH^\infty_t}_{t \in [0,\infty)}
	,&&	
	\text{ where } &&\cH^\infty_t := \bigvee_{k \geq 0} \cF^{N^k}_t,
	\\
	\widehat{\mathbb{H}}^{k+1} &= \set{\wh{\cH}^{k+1}_t}_{t \in [0,\infty)},
		&&
	\text{ where } &&\wh{\mathcal{H}}^{k+1}_t := \cH^k_\infty \vee \cF^{N^{k+1}}_t.
	\end{align*}
	
Our first aim is to compute  $\bH^\infty$-compensator of the limiting random measure $H^{\infty}$ given in \eqref{Hinf}. We begin with following key result,

\begin{proposition}\label{prop:fdsmpp}
i) The marked point process $N^0$ is an  $\bH^0$-doubly stochastic marked Poisson process.
	The random measure $\nu^0$ given by
	\begin{equation}\label{eq:nu0-intensity}
	\nu^{0}((s,t] \times D) = \int_s^t \eta(v,D) dv,\qquad 0\leq s\leq t,\ D\in \cE^\Delta,
	\end{equation} is the $\cH^0_0$-intensity kernel of $N^0$.  Moreover,  $\nu^0$ is the $\bH^0$-compensator of $N^0$.

	\noindent ii) For each j the marked point process $N^{k+1}_j$ is an  $\wh{\mathbb{H}}^{k+1}$-doubly stochastic marked Poisson process.
The random measure $\nu^{k+1}_j$ given by 	
\begin{equation}\label{eq:nu-intensity}
\nu^{k+1}_j((s,t] \times D) = \int_s^t f(v, S^k_j, Y^k_j, D) \1_{{(} S^k_j, \infty )} (v) dv,\qquad 0\leq s\leq t,\ D\in \cE^\Delta,
\end{equation}
 is the $\wh{\cH}^{k+1}_0$-intensity kernel of $N^{k+1}_j$. Moreover,  $\nu^0$ is the $\wh{\mathbb{H}}^{k+1}$-compensator of $N^{k+1}_j$.
\end{proposition}
\begin{proof}
	i). Note that from Lemma  \ref{lem:Gcprm}, by taking
	\[
	\cG = \cH^0_0, \quad \cY = \set{1} , \quad Y = 1 , \quad \ell (y) = 0,  \quad g = \eta,
	\]
	it follows that $N^0$ is  $\cH^0_0$-conditional Poisson random measure with intensity measure $\nu^0$ given by \eqref{eq:nu0-intensity}. 
	Now, the  assertion follows from the point  i) of Proposition \ref{prop:fdsmpp2}.

	\medskip
	\noindent ii).  We first note that from Lemma \ref{lem:Gcprm}, by taking
	\[
	\cG = \cH^k_\infty, \quad \cY = [0,T] \times E^\Delta , \quad Y = (S^k_j, Y^k_j)  , \quad \ell (s,y) = s,  \quad g = f,
	\]
	it follows that $N^{k+1}_j$ defined by \eqref{eq:Nkp1j} is $\cH^k_\infty$-conditionally Poisson random measure with intensity measure $\nu^{k+1}_j$ given by \eqref{eq:nu-intensity}.

	To complete the proof, in view of assertion ii) of Proposition \ref{prop:fdsmpp2}, it suffices show that the marked point processes $(N^{k+1}_j)_{j \geq 1}$ are  $\wh{\mathcal{H}}^{k+1}_0$-conditionally independent.
	Since $\wh{\cH}^{k+1}_0=\cH^k_\infty = \sigma(N^0, \ldots, N^k)$ it suffices to verify that $(N^{k+1}_j)_{j \geq 1} $ are conditionally independent given $\sigma(N^0, \ldots, N^k)$.  For this we first note that for each $j$ the random measure   $N^{k+1}_j$ is defined by \eqref{eq:Nkp1j}, so it is constructed from the pair $(S^k_j, Y^k_j)$,  which is $ \sigma(N^0, \ldots, N^k) $-measurable and from  the family

\[I_j := (Z^{k+1,j},(U^{k+1,j}_i, V^{k+1,j}_i, W^{k+1,j}_i)_{i=1}^\infty).\]

Now, using the fact that $I_1, I_2, \ldots $ are independent between themselves and also independent from $ \sigma(N^0, \ldots, N^k) $, we conclude that $N^{k+1}_1,N^{k+1}_2, \ldots$ are $(N^0, \ldots, N^k)$-conditionally independent.
	So we see that  $N^{k+1}_{j}$ is a $\wh{\cH}^{k+1}_0$-conditional Poisson random measure for any $j \geq 1$, and that $(N^{k+1}_{j})_{j \geq 1}$ are $\wh{\cH}^{k+1}_0$-conditionally independent random measures.  Thus we may use
	Proposition \ref{prop:fdsmpp2} to conclude that $N^{k+1}_j$ is an  $\wh{\mathbb{H}}^{k+1}$-doubly stochastic marked Poisson process whose  $\wh{\cH}^{k+1}_0$-intensity kernel is   $\nu^{k+1}_j$ given by \eqref{eq:nu-intensity}
	\end{proof}

From Proposition \ref{prop:fdsmpp} and from its proof
we conclude  that
the random measure $N^{k+1}$ given by
 \eqref{eq:Nkp1}
is a sum of  $\cH^k_\infty$-conditionally independent $\wh{\bH}^{k+1}$-doubly stochastic marked Poisson processes. We will prove now that  $N^{k+1}$ is also an $\wh{\bH}^{k+1}$-doubly stochastic marked Poisson process whose intensity kernel is simply the sum of intensity kernels of $N^{k+1}_j$, $j \geq 0$.
\begin{proposition}\label{prop:nukp1}
	The marked point process $N^{k+1}$ is an $\wh{\bH}^{k+1}$-doubly stochastic marked Poisson process
		with intensity kernel $\nu^{k+1}$  given by
	\begin{equation}\label{eq:nu-intensity-k+1}
	\nu^{k+1}((s,t] \times D) = \sum_{j =1}^\infty \nu^{k+1}_j ((s,t] \times D)= \int_s^t \int_{(0,v) \times E^\Delta} f(v, u, y, D) N^k(du,dy) dv,
	\end{equation}
	for $0 \leq s \leq t$, $D \in \cE^\Delta$.
	Moreover, the intensity kernel $\nu^{k+1}$ of  $N^{k+1}$ is  the $\mathbb{H}^{k+1}$-compensator of  $N^{k+1}$.
\end{proposition}
\begin{proof}
	To prove the first assertion,
	in view of Proposition 6.1.4 in \cite{LasBra1995}, it suffices to show that $\nu^{k+1}$ is the $\wh{\bH}^{k+1}$-compensator of $N^{k+1}$. Indeed this compensating property implies  that
	\[
	\bE( N^{k+1}((s,t] \times B) | \wh{\cH}^{k+1}_0 ) = \bE(  \nu^{k+1}((s,t] \times B) | \wh{\cH}^{k+1}_0 ) = \nu^{k+1}((s,t] \times B),
	\]
	where the last equality follows from $\wh{\cH}^{k+1}_0$-measurability of $\nu^{k+1}$. So, if $\nu^{k+1}$ is $\wh{\bH}^{k+1}$-compensator of $N^{k+1}$ then it is $\wh{\cH}^{k+1}_0$-intensity kernel and, thus,
	Theorem 6.1.4 in \cite{LasBra1995} implies the first assertion. Therefore it remains to show that $\nu^{k+1}$ is $\wh{\bH}^{k+1}$-compensator of $N^{k+1}$.
	
Towards this end we first note that from Proposition \ref{prop:fdsmpp} it follows that $N^{k+1}_j$ is an $\wh{\bH}^{k+1}$-doubly stochastic marked Poisson process with $\wh{\bH}^{k+1}$-compensator  $\nu^{k+1}_j$ given by \eqref{eq:nu-intensity}.
	So, for an arbitrary non-negative  $\wh{\mathbb{H}}^{k+1}$-predictable function $F:\Omega \times [0,T] \times {E}^\Delta \rightarrow \bR$ it holds
	\begin{equation}\label{eq:CN1}
	\bE \Big( \int_0^\infty \int_{E^\Delta} F(u,y) N^{k+1}_j(du,dy) \Big)
	=\bE \Big( \int_0^\infty \int_{E^\Delta} F(u,y) \nu^{k+1}_j(du,dy) \Big), \qquad j \in \bN.
	\end{equation}	
	This implies that
	\begin{align*}
	\lim_{m\rightarrow \infty} \bE \Big( \int_0^\infty \int_{E^\Delta} F(u,y) \Big(\!\sum_{j=1}^m N^{k+1}_j\!\Big)(du,dy) \Big)
	=\lim_{m\rightarrow \infty}\bE \Big( \int_0^\infty \int_{E^\Delta} F(u,y) \Big( \sum_{j=1}^m \nu^{k+1}_j\Big)(du,dy) \Big).
	\end{align*}	
	Since, for every $A \in \cB([0, \infty)) \otimes \cE^\Delta$,
	\[
	{N}^{k+1}(\omega, A) = \lim_{m \uparrow \infty} \sum_{j=1}^m {N}^{k+1}_{j}(\omega, A), \quad  \text{ and } \quad
	{\nu}^{k+1}(\omega, A)  = \lim_{m \uparrow \infty} \sum_{j=1}^m {\nu}^{k+1}_{j}(\omega, A)  \]
	 almost surely,
	using Lemma \ref{lem:NCT} for
	\[
		\mu_j( dt, dy, d \omega) = N^{k+1}_j(\omega, dt,dy)  \bP(d\omega),\quad
		\mu( dt, dy, d \omega) = N^{k+1}(\omega, dt,dy)  \bP(d\omega)
	\]
	and once again for
	\[
\overline{\mu}_j( dt, dy, d \omega) = \nu^{k+1}_j(\omega, dt,dy)  \bP(d\omega),\quad
\overline{\mu}( dt, dy, d \omega) = \nu^{k+1}(\omega, dt,dy)  \bP(d\omega)
\]
we see that
	\begin{align*}
	\bE \Big( \int_0^\infty \int_{E^\Delta} F(u,y) N^{k+1}(du,dy) \Big)
	=		\bE \Big( \int_0^\infty \int_{E^\Delta} F(u,y)  \nu^{k+1}(du,dy) \Big).
	\end{align*}
	Now, since	for any $ 0 \leq s \leq t$, $D \in \cE^{\Delta}$
	\begin{align*}
		\nu^{k+1}((s,t] \times D) &=
		\sum_{j=1}^\infty \nu^{k+1}_j((s,t] \times D)
		=\sum_{j=1}^\infty\int_s^t f(v, S^k_j, Y^k_j, D) \1_{( S^k_j, \infty )}(v) dv
		\\
		&=\int_s^t \sum_{j=1}^\infty f(v, S^k_j, Y^k_j, D) \1_{( S^k_j, \infty )}(v) dv,
	\end{align*}		
	and since
	\[
	\int_{(0,v) \times E^\Delta} f(v, u, y, D) N^{k}(du,dy)
	=\sum_{j=1}^\infty f(v, S^k_j, Y^k_j, D) \1_{( S^k_j, \infty )}(v) ,
	\]
we obtain that \eqref{eq:nu-intensity-k+1} holds. This concludes the proof of the first assertion.

Now we will prove that  the {$\wh{\cH}^{k+1}_0$-}intensity kernel $\nu^{k+1}$ of  $N^{k+1}$ is  the $\mathbb{H}^{k+1}$-compensator of  $N^{k+1}$. For this, we first observe that from Theorem 6.1.4 in \cite{LasBra1995} it follows that the intensity kernel of  $N^{k+1}$ is the $\wh{\mathbb{H}}^{k+1}$-compensator of  $N^{k+1}$.
So, for an arbitrary non-negative  $\wh{\mathbb{H}}^{k+1}$-predictable function  $F:\Omega \times [0,T] \times {E}^\Delta \rightarrow \bR$ it holds
\begin{equation}\label{eq:CN}
\bE \Big( \int_0^\infty \int_{E^\Delta} F(u,y) N^{k+1}(du,dy) \Big)
=\bE \Big( \int_0^\infty \int_{E^\Delta} F(u,y) \nu^{k+1}(du,dy) \Big).
\end{equation}	
Since $\wh{\mathbb{H}}^{k+1} \supset \mathbb{H}^{k+1}$, the $\mathbb{H}^{k+1}$-predictable
functions  are also $\wh{\mathbb{H}}^{k+1}$-predictable. So \eqref{eq:CN} holds for an arbitrary non-negative $\mathbb{H}^{k+1}$-predictable function $F$.
From \eqref{eq:nu-intensity-k+1} we see that for an arbitrary $D$ the process $(\nu^{k+1}((0,t] \times D))_{t\in [0,T]}$ is $\bH^{k}$-adapted and continuous. Hence it is  $\bH^{k}$-predictable and thus also $\bH^{k+1}$-predictable (since $\bH^{k+1} \supset \bH^{k}$). So, $\nu^{k+1}$ is an $\bH^{k+1}$-predictable random measure such that \eqref{eq:CN} holds for arbitrary non-negative $\mathbb{H}^{k+1}$-predictable function $F$. This means that $\nu^{k+1}$ is the  $\mathbb{H}^{k+1}$-compensator of $N^{k+1}$.
\end{proof}

In order to proceed we will need the following auxiliary result.
\begin{lemma}\label{lem:immersion}
	Let $\bF$ and $\bG$ be  filtrations in $(\Omega,\cF,\bP)$. Then	
$\bF$ is $\bP$-immersed in $\bF \vee \bG$ if and only if
for every $t \geq 0$  and every bounded $\cG_t$-measurable random variable $\eta$ we have
			\begin{equation}\label{eq:etka}
			\bE( \eta | \cF_t) = \bE( \eta | \cF_\infty) .
			\end{equation}
\end{lemma}
\begin{proof}

The  necessity follows from Proposition 5.9.1.1 in \cite{JeaYorChe2009}.
To prove sufficiency it is enough to show, again by Proposition 5.9.1.1 in \cite{JeaYorChe2009}, that for  every $t \geq 0$ and  every bounded $\cF_\infty$-measurable random variable $\xi$  it holds that
\[
\bE(\xi| \cF_t \vee \cG_t) =\bE(\xi| \cF_t ).
\]
%
Fix $\xi$, we need to show that
	\begin{equation}\label{eq:cosik}
		\bE( \xi \1_A ) = \bE( \bE(\xi | \cF_t) \1_A ),
	\end{equation}
	for every $A \in \cF_t \vee \cG_t$.
	Towards this end let us consider a family $\cU$ of sets defined as
	\[
		\cU = \set{ A: A=B\cap C,  B \in \cF_t, C \in \cG_t}.
	\]
	Note that $\cU$ is a $\pi$-system of sets which generates $\cF_t \vee  \cG_t$.
	Observe that family of all sets for which \eqref{eq:cosik} holds is a $\lambda$-system.
	Thus, by the Sierpinski's Monotone Class Theorem (cf. Theorem 1.1 in \cite{Kal2002}), which is also known as the Dynkin's $\pi-\lambda$ Theorem,  it suffices to prove \eqref{eq:cosik} for the sets from $\cU$, which we will do now so to complete the proof.

	For $A \in \cU$, we have
	\begin{align*}
		\bE( \xi \1_A ) & = \bE( \xi \1_{B\cap C} )  = \bE( \bE( \xi \1_{B\cap C} | \cF_\infty))
		 =\bE( \xi \1_B \bE(  \1_{ C} | \cF_\infty))
		 =\bE( \xi \1_B \bE(  \1_{ C} | \cF_t))
		\\
		&=\bE( \xi  \bE(  \1_{ B \cap C} | \cF_t)) =\bE( \bE(  \xi  \bE(  \1_{ B \cap C} | \cF_t)| \cF_t))
		=\bE( \bE(  \xi  | \cF_t) \bE(  \1_{ B \cap C} | \cF_t))
		\\
		&
		=\bE( \bE(  \bE(  \xi  | \cF_t)  \1_{ B \cap C} | \cF_t))
		= \bE(  \bE(  \xi  | \cF_t)  \1_{ B \cap C} )
		= \bE(  \bE(  \xi  | \cF_t)  \1_A ),
	\end{align*}
	where the fourth equality follows from \eqref{eq:etka} for $\eta = \1_C$.
\end{proof}

We are now ready to demonstrate the following proposition.

\begin{proposition}\label{prop:immersion}
	Filtration $\bH^{k}$ is $\bP$-immersed in  $\bH^{k+1}$.
\end{proposition}
\begin{proof}
	Since $\bH^{k+1}  = \bH^{k} \vee \bF^{N^{k+1}}$ we use  Lemma  \ref{lem:immersion} to prove immersion of $\bH^{k}$ in  $\bH^{k} \vee \bF^{N^{k+1}}$.
	It suffices to show that
			\begin{equation}\label{eq:immersja}
	\begin{aligned}
	\bP(A | \cH^{k}_\infty)
	=
	\bP( A | \cH^{k}_u),
	\end{aligned}
	\end{equation}
for every $u \geq 0 $ and every  $ A \in \cU$,
	where
	\begin{align*}
	\cU = \bigg\{ A : & A= \bigcap_{i=1}^n\set{ N^{k+1}((s_i, t_i] \times D_i )=l_i }, D_1, \ldots, D_n \text{ are disjoint sets, and }
		\\
		&
		0 \leq s_1 < t_1 \leq s_2 < t_2 \leq \ldots \leq s_n < t_n \leq u, n \in \bN \bigg\}.
	\end{align*}
	Indeed, if \eqref{eq:immersja}  holds for $ A \in \cU$, then since $\cU$ is a $\pi$-system which generates $\cF^{N^{k+1}}_u$ the monotone class theorem implies that \eqref{eq:etka} holds.
	It remains to show \eqref{eq:immersja} for $ A \in \cU$. Using Proposition \ref{prop:nukp1} and invoking \eqref{eq:poissonrm} we have 	
	\begin{align}\label{eq:Nkp1a}
		 \bP\Big(\bigcap_{i=1}^n\set{ N^{k+1}((s_i, t_i] \times D_i )=l_i } | \wh{\cH}^{k+1}_0 \Big)
		 =
		 \prod_{i=1}^ne^{ -\nu^{k+1} ((s_i, t_i] \times D_i)} \frac{(\nu^{k+1} ((s_i, t_i] \times D_i))^{l_i+1}}{l_i!}.
	\end{align}
	Since $\wh{\cH}^{k+1}_0 = \cH^k_\infty $
	and $\nu^{k+1} ((s_i, t_i] \times D)$ is $\cH^{k}_{t_i}$
	measurable we infer that  the right hand side of \eqref{eq:Nkp1a}
	is $\cH^k_{t_n}$-measurable and hence also $\cH^k_{u}$-measurable  for arbitry $u \geq t_n$. Consequently by taking conditional expectations with respect to $\cH^k_u$ for $u \geq t_n$ we conclude that \eqref{eq:immersja} holds for $A \in \cU$. The proof is complete.
\end{proof}

{We will determine now the compensators for $H^0:= N^0$ and for  $H^k$ given by  \eqref{eq:Hk} for $k \geq 1$.}
\begin{proposition}\label{prop:hinftycomp}
	The $\bH^{0}$-compensator of $H^{0}$, is given by
	\[
			\vartheta^{0}((s,t] \times D) = \int_s^t \eta(v,D) dv,\qquad 0\leq s\leq t,\ D\in \cE^\Delta,
	\]
	where
	the kernel $\eta$ appears in \eqref{new-nu}.

	The $\bH^{k}$-compensator of $H^{k}$, for $k \geq 1$, is given by
\begin{equation}\label{eq:kompHk}
	 	\vartheta^{k}((s,t] \times D) = \int_s^t \Big(\eta(v,D) +   \int_{(0,v) \times E^\Delta} f(v, u, y, D) H^{k-1}(du,dy) \Big) dv,\qquad 0\leq s\leq t,\ D\in \cE^\Delta.
\end{equation}
\end{proposition}
\begin{proof}
	The proof goes by induction.

	{Since $H^{0} = N^0$, then the form of  $\bH^{0}$-compensator of $H^{0}$ follows from assertion i) of Proposition \ref{prop:fdsmpp} and from  Proposition 6.1.4 \cite{LasBra1995}. }

	Suppose now that $\bH^k$-compensator of $H^k$ is given by \eqref{eq:kompHk}.
	This means that for every $D \in \cE^\Delta$ the process
	\begin{equation}
	M^{k}_t(D) = (H^{k}- \vartheta^{k})((0,t] \times D) ,  \quad t \geq 0,
	\end{equation}
	is an $\bH^k$-local martingale. Proposition \ref{prop:immersion}  implies that $M^k(D)$ is an $\bH^{k+1}$-local martingale.
	 We know from Proposition \ref{prop:nukp1} that
	\[
		L^{k+1}_t(D) = (N^{k+1}- \nu^{k+1})((0,t] \times D) ,  \quad t \geq 0,
	\]
	is an $\bH^{k+1}$-local martingale. Thus
		$M^{k}(D) + L^{k+1}(D)$ is an $\bH^{k+1}$-local martingale. This $\bH^{k+1}$-local martingale can be written in the form
	\[
		M^{k}_t(D) + L^{k+1}_t(D) = (H^k + N^{k+1} - ( \vartheta^k + \nu^{k+1}))((0,t] \times D)
	 =(H^{k+1} - ( \vartheta^k + \nu^{k+1}))((0,t] \times D),
	\]
	where the second equality follows from
	\[
	H^{k+1} = H^k + N^{k+1}.
	\]
	Note that the random measure $\vartheta^k + \nu^{k+1}$ is $H^{k+1}$-predictable so it is the
	$\bH^{k+1}$-compensator of $H^{k+1}$. To complete the proof it suffices to show that
	$\vartheta^k + \nu^{k+1} = \vartheta^{k+1}$. By the induction hypothesis on $\vartheta^k$ and by \eqref{eq:nu-intensity-k+1} we have
	\begin{align*}
	(\vartheta^k + \nu^{k+1})((s,t] \times D)
	&=
	\int_s^t \Big(\eta(v,D) +   \int_{(0,v) \times E^\Delta} f(v, u, y, D) (H^{k-1} + N^{k})(du,dy) \Big) dv
	\\
	&=\vartheta^{k+1}((s,t] \times D).
	\end{align*}
	The proof is complete.
\end{proof}

Before we conclude our construction of a generalized multivariate Hawkes process, we derive the following result.
\begin{proposition}
	The $\bH^\infty$-compensator of $H^\infty$ is given by
	\begin{equation}\label{eq:cudak}
		\vartheta^{\infty}((s,t] \times D)
		=
		\int_s^t \Big(\eta(v,D) +   \int_{(0,v) \times E^\Delta} f(v, u, y, D) H^{\infty}(du,dy) \Big) dv.
	\end{equation}
\end{proposition}
\begin{proof}
	Proposition  \ref{prop:immersion} and Proposition \ref{prop:hinftycomp} imply that for every $k \geq 1$, the $\bH^{\infty}$-compensator of $H^{k}$ is given by \eqref{eq:kompHk}. Thus, we see that for any $k \geq 1$ and for an arbitrary non-negative $\bH^\infty$-predictable function  $F:\Omega \times [0,T] \times E^\Delta \rightarrow \bR$ it holds
	\begin{equation*}
	\bE \Big( \int_0^\infty \int_{E^\Delta} F(v,y) H^{k}(dv,dy) \Big)
	=\bE \Big( \int_0^\infty \int_{E^\Delta} F(v,y) \vartheta^{k}(dv,dy) \Big).
	\end{equation*}	
	Using Lemma \ref{lem:NCT} in an  analogous way as in the proof of Proposition \ref{prop:nukp1} we obtain
\begin{equation}\label{eq:CHlim}
\bE \Big(  \int_0^\infty \int_{E^\Delta} F(u,y) H^{\infty}(du,dy) \Big)
=
\bE \Big( \int_0^\infty \int_{E^\Delta} F(u,y) \vartheta^{\infty}(du,dy) \Big).
\end{equation}
This completes the proof.
\end{proof}

We are now ready to conclude our construction of a generalized multivariate Hawkes process. Let $T_\infty$ be the first accumulation time of $H^\infty$.\footnote{$T_\infty =\lim_{n\rightarrow \infty} T^{\infty}_n$, where  $T^{\infty}_n:= \inf \set{ t: H^{\infty}((0,t] \times E^\Delta)  \geq n}$.}  Then we have the following

\begin{theorem}
	The process $N:=\1_{]\!]0,T_\infty[\![}H^\infty $ is an $\bF^N$-Hawkes process with the Hawkes kernel\footnote{We recall our notational convention that $\eta =\eta_T:= \1_{[0,T]}\eta$ and $f =f_T:= \1_{[0,T]}f$.}
\[ \kappa(t,dx) =  \eta(t,dx) +   \int_{(0,t) \times E^\Delta} f(t, u, y, dx)N(du,dy).\]
\end{theorem}
\begin{proof}
	Let us define a sequence $(T_n,X_n)_{n \geq 1}$ by
	\begin{align*}
	T_n &= \inf \set{ t >0: H^\infty((0,t] \times E^\Delta)  \geq n},
	\\
	X_n &=
	{
		\beta^{-1}
		\Big(
		\1_\set{ T_n < \infty }
		\int_{E^\Delta} \beta(x) H^\infty(\set{T_n} \times dx)
		+ \widehat\partial \1_\set{ T_n = \infty } \Big)},
	\end{align*}
	and the random measure
	\[
		N(dt,dx) = \sum_{n>0 } \delta_{(T_n,X_n)} (dt,dx) \1_\set{T_n < \infty}.
	\]
	Then
	\[
		N(dt,dx) = H^\infty(dt, dx) |_{]\!]0, T_\infty[\![ \times E^\Delta}.
	\]
	Consequently, such restriction of $H^\infty$ to $]\!]0, T_\infty[\![ \times E^\Delta$ is a marked point process.
	Moreover, since $]\!]0, T_\infty[\![ $ is an $\bH^\infty$-predictable set, we have for arbitrary non negative $\bH^\infty$-predictable function  $F:\Omega \times \bR_+ \times E^\Delta \rightarrow \bR$
\begin{align} \nonumber
	&
		\bE \Big( \int_0^\infty \int_{E^\Delta} F(u,x) N(du,dx) \Big) =
		\bE \Big( \int_0^\infty \int_{E^\Delta} F(u,x) \1_{]\!]0, T_\infty[\![ \times E^\Delta}  H^\infty(du,dx) \Big)
	\\ \label{eq:Hinftycomp}
	&	=
		\bE \Big( \int_0^\infty \int_{E^\Delta} F(u,x) \1_{]\!]0, T_\infty[\![ \times E^\Delta}  \vartheta^\infty(du,dx) \Big),
\end{align}
where $\vartheta^\infty$ is given in \eqref{eq:cudak}.

So the compensator of the  restriction of  $H^{\infty}$ to $]\!]0, T_\infty[\![ \times E^\Delta$ is the restriction to $]\!]0, T_\infty[\![ \times E^\Delta$ of  compensator of $H^{\infty}$.
Now we will prove
that
\[
\vartheta^\infty(dt,dx)|_{]\!]0, T_\infty[\![ \times E^\Delta}=
\1_{(0, T_\infty(\omega))}(t) \kappa(t,dx) dt,
\]
where
\[
\kappa(t,dx) =\eta(t,dx) +   \int_{(0,t) \times E^\Delta} f(t, u, y, dx) N(du,dy).
\]
Towards this end note that for arbitrary  $0 \leq s  < t \leq T$ and  $D \in \cE^\Delta$ we have
\begin{align*}
&\vartheta^\infty|_{]\!]0, T_\infty[\![ \times E^\Delta}((s,t] \times D)
=\vartheta^{\infty}((s,t] \cap (0, T_\infty(\omega))  \times D)
\\
&=
\int_s^t \1_{(0, T_\infty(\omega))} (v)\Big(\eta(v,D) +   \int_{(0,v) \times E^\Delta} f(v, u, y, D) H^{\infty}(du,dy) \Big) dv.
\end{align*}
The second term above can be written as
\begin{align*}
&\1_{(0, T_\infty(\omega))} (v)
 \int_{(0,v) \times E^\Delta} f(v, u, y, D) H^{\infty}(du,dy)
\\
& =
\int_{(0,T] \times E^\Delta} \1_{(0,v)}(u) \1_{(0, T_\infty(\omega))} (v) f(v, u, y, D) H^{\infty}(du,dy)
\\
&=
\int_{(0,T] \times E^\Delta} \1_\set{u < v < T_\infty(\omega) } \1_{(0, T_\infty(\omega))}(u)  f(v, u, y, D) H^{\infty}(du,dy)
\\
&=
\1_{(0,T_\infty(\omega)) } (v)\int_{(0,T] \times E^\Delta} \1_\set{u < v}
 f(v, u, y, D) \1_{(0, T_\infty(\omega))}(u)  H^{\infty}(du,dy)
 \\
 &
 =
 \1_{(0,T_\infty(\omega)) } (v)\int_{(0,v) \times E^\Delta}
 f(v, u, y, D) N(du,dy).
\end{align*}
Hence
\begin{align*}
&\vartheta^\infty|_{]\!]0, T_\infty[\![ \times E^\Delta}((s,t] \times D)
\\
&=
\int_s^t \1_{(0, T_\infty(\omega))} (v)\Big(\eta(v,D) +   \int_{(0,v) \times E^\Delta} f(v, u, y, D) N(du,dy) \Big) dv.
\\
&=
\int_s^t \1_{(0, T_\infty(\omega))} (v)\kappa(v,D) dv.
\end{align*}
This and
\eqref{eq:Hinftycomp}  imply that $\vartheta^\infty|_{]\!]0, T_\infty[\![ \times E^\Delta}$ is an $\bF^N$-predictable random measure
such that for arbitrary non negative 	$\bF^N$-predictable function $F:\Omega \times \bR_+ \times E^\Delta \rightarrow \bR$  we have
\begin{align} \nonumber
&
\bE \Big( \int_0^\infty \int_{E^\Delta} F(u,x) N(du,dx) \Big) =
\bE \Big( \int_0^\infty \int_{E^\Delta} F(u,x) \1_{]\!]0, T_\infty[\![ } (u) \kappa(u,dx) du \Big).
\end{align}
Thus $N$ is a $\bF^N$-Hawkes process (restricted to $[0,T]\times E^\Delta$) with the Hawkes kernel $\kappa$.
\end{proof}

\subsection{The pseudo-algorithm}

In the description of the pseudo-algorithm below we use the objects $\eta$, $f$, $\wh \eta$, $\wh f$, $G_1$ and $G_2$ that underly the construction of our Hawkes process given in Section \ref{sec:constr}.

The steps of the pseudo-algorithm are based on the steps presented in our construction of a generalized multivariate Hawkes process with deterministic kernels $\eta$ and $f$, and  they are:
\begin{itemize}
	\item[Step 0.] Choose a positive integer $K$, set  $C^0 = \varnothing $.
	\item[Step 1.] Generate a realization, say $p$,  of a Poisson random variable with parameter $T \wh \eta $.
	\item[Step 2.] If $p=0$, then go to Step 3. Else, if $p> 0$, then  for $i=1, \ldots , p\, $:
	\begin{itemize}
		\item Generate  realizations $u$ and $v$  of independent  random variables uniformly distributed on $[0,1]$. Set
	$  t = Tu,  a =\hat{\eta} $.
	\item
	If $a \leq  \eta(t, E^\Delta)$, then generate a realization $w$ of random variable  uniformly distributed on $[0,1]$, compute $x = G_1(t, w)$ and  include  $(t,x)$ into the cluster $C^0$.
	\end{itemize}

	\item[Step 3.]
	Set  $\mathcal{N} = C^0$, $C^{prev} = C^0$, $k = 0$. 
	
	\item[Step 4.]  While $C^{prev} \neq \varnothing$ and  $k \leq K\, $:
					\begin{itemize}
					\item Set $C^{new} = \varnothing$ .
					\item For every $(s,y)
					\in C^{prev}$:
					\begin{itemize}
					\item  generate a realization $p$  of Poisson random variable with parameter\\ $(T-s)\widehat{f}(s,y)$.
					\item  for $i=1, \ldots, p$:
					\begin{itemize}
					\item[$\diamond$] Generate  realizations $u$ and $v$  of independent  random variables uniformly distributed on $[0,1]$ . Set
					$  t = s + (T-s)u,  a =\widehat{f}(s,y)v $.
					\item[$\diamond$]
					If $a \leq  f(t,s,y, E^\Delta)$, then generate a realization $w$ of random variable  uniformly distributed on $[0,1]$, compute $x = G_2(t,s,y, w)$ and  include  $(t,x)$ into the cluster $C^{new}$.
					\end{itemize}
					\end{itemize}
						\item Set $ \mathcal{N} =  \mathcal{N} \cup C^{new}$, $C^{prev} = C^{new}$. 						
						\item Set $k= k +1$.
					\end{itemize}
		\item[Step 5.] Return $ \mathcal{N}$.
\end{itemize}

\subsubsection{Numerical examples via simulation}

The pseudo-algorithm presented above is implemented here  in two cases. In the first case, presented in Example \ref{1}, we implemented the algorithm for a  generalized bivariate Hawkes process with $E_1=E_2 = \set{1}$. In the second case,  presented in Example \ref{2}, we set  $E_1=E_2 = \bR$.

We used Python to run the simulations and to plot graphs.

\begin{example}\label{1}$\ $ {\bf Bivariate point Hawkes process}
\medskip

Here we implement our pseudo-algorithm for a bivariate point Hawkes process, that is the generalized bivariate Hawkes process with $E_1=E_2 = \set{1}$, and hence with \[ E^\Delta = \set{ (1, \Delta), (\Delta,1), (1,1)}. \]
Moreover, we let
\[
\eta(t,dy):=
\eta_{1}(t)
\delta_{(1,\Delta)} (d y)
+
\eta_{2}(t)\delta_{(\Delta,1)} (d y)
+
\eta_{c}(t)\delta_{(1,1)} (d y),
\]
where
\[
\eta_i(t) := \alpha_i  + ( \eta_i(0) - \alpha_i ) e^{-\beta_i t}, \qquad i\in \{1,2,c\},
\]
and $\alpha_i, \eta_i(0), \beta_i$ are non-negative constants.
We assume that, for $0\leq s\leq t$, the kernel $f$ is given as in \eqref{eq:f-Hawkes} with the decay functions $w_{i,j}$  in the exponential form:
\[
w_{i,j}(t,s) = e^{-\beta_i(t-s)}, \qquad i,j \in \{1,2,c\},
\]  with constant non-negative impact functions:
\begin{align*}
g_{1,1}(x_1)=\vartheta_{1,1}, \quad
g_{1,2}(x_2)=\vartheta_{1,2}, \quad
g_{1,c}(x)=\vartheta_{1,c}, \\
g_{2,1}(x_1)=\vartheta_{2,1}, \quad
g_{2,2}(x_2)=\vartheta_{2,2}, \quad
g_{2,c}(x)=\vartheta_{2,c},\\
g_{c,1}(x_1)=\vartheta_{c,1}, \quad
g_{c,2}(x_2)=\vartheta_{c,2}, \quad
g_{c,c}(x)=\vartheta_{c,c},
\end{align*}
and with Dirac kernels:
\[
\phi_1(x,dy_1) = \delta_1(dy_1), \quad
\phi_2(x,dy_2) = \delta_1(dy_2), \quad
\phi_c(x,dy_1, dy_2) = \delta_{(1,1)}(dy_1, dy_2).
\]
Thus, the kernel $f$ is of the form
\begin{align}
f&(t,s,x,dy) \label{eq:kernel-exp}
\\ \nonumber
=&  \  e^{-\beta_1(t-s)} \Big(  \vartheta_{1,1}\1_{\set{1} \times\Delta } (x)  + \vartheta_{1,2} \1_{\Delta \times \set{1}} (x)
+ \vartheta_{1,c} \1_{\set{1} \times \set{1} } (x) \Big) \delta_{(1,\Delta)} (d y)   \\
\nonumber
&+  e^{-\beta_2(t-s)} \Big(  \vartheta_{2,1}  \1_{\set{1} \times\Delta } (x)  +  \vartheta_{2,2}\1_{\Delta \times \set{1}} (x)
+  \vartheta_{2,c} \1_{\set{1} \times \set{1} } (x)\Big) \delta_{(\Delta,1)} (d y) \\ \nonumber
&+  e^{-\beta_c(t-s)}
\Big(  \vartheta_{c,1} \1_{\set{1} \times\Delta } (x) +  \vartheta_{c,2} \1_{\Delta \times \set{1}} (x)
+ 	 \vartheta_{c,c} \1_{\set{1} \times \set{1} } (x)  \Big) \delta_{(1,1)} (d y).
\end{align}	
The coordinates of $N$ (cf. \eqref{N-i}) reduce here to counting (point) processes, so that
\begin{equation*}
N^1_t  = N^1((0,t],\set{1})=N((0,t],\set{1} \times \set{1, \Delta}),
\end{equation*}
and
\begin{equation*}
N^2_t = N^2((0,t],\set{1})=N((0,t],\set{1, \Delta} \times \set{1}  ).
\end{equation*}
Moreover, $N^{idio, \set{1,2} }$ -- the MPP of idiosyncratic group of $\set{1,2}$ coordinates -- reduces here to the  process counting the number of occurrences of the common events:
\begin{equation*}
N^{c}_t = N^{idio, \set{1,2} }((0,t],\set{(1,1)})=N((0,t],\set{(1, 1)}   ).
\end{equation*}
We take the following values of parameters:
\[
\begin{array}{|c|c|c|c|}
	\hline
	i & \eta_i(0) & \alpha_i &  \beta_i \\ \hline \hline
	1 & 0.5 & 0.5  & 2.5 \\ \hline
	2 & 0.5 & 0.5 &  2.5 \\ \hline
	c & 0.25 & 0.25 & 5.0 \\ \hline
\end{array}\qquad
\begin{array}{|lc|ccc|}
\hline
\vartheta_{i,j} &  &                       &       j                &  \\ \cline{3-5}
&  & \multicolumn{1}{c|}{1} & \multicolumn{1}{c|}{2} & c \\ \hline
\multicolumn{1}{|l|}{} & 1 & \multicolumn{1}{l|}{0.5} & \multicolumn{1}{l|}{0.25} & 0.25 \\ \cline{2-5}
\multicolumn{1}{|r|}{i} & 2 & \multicolumn{1}{l|}{0.25} & \multicolumn{1}{l|}{0.5} &  0.25 \\ \cline{2-5}
\multicolumn{1}{|l|}{} & c & \multicolumn{1}{l|}{0.25} & \multicolumn{1}{l|}{0.25} & 0.25 \\ \hline
\end{array}.
\]
Simulated sample paths of $N$  corresponding to the above setting  are presented in Figure 1 and Figure 2.

\begin{figure}[h]
\hspace*{-2cm}	
	\includegraphics[width=1.2\textwidth]{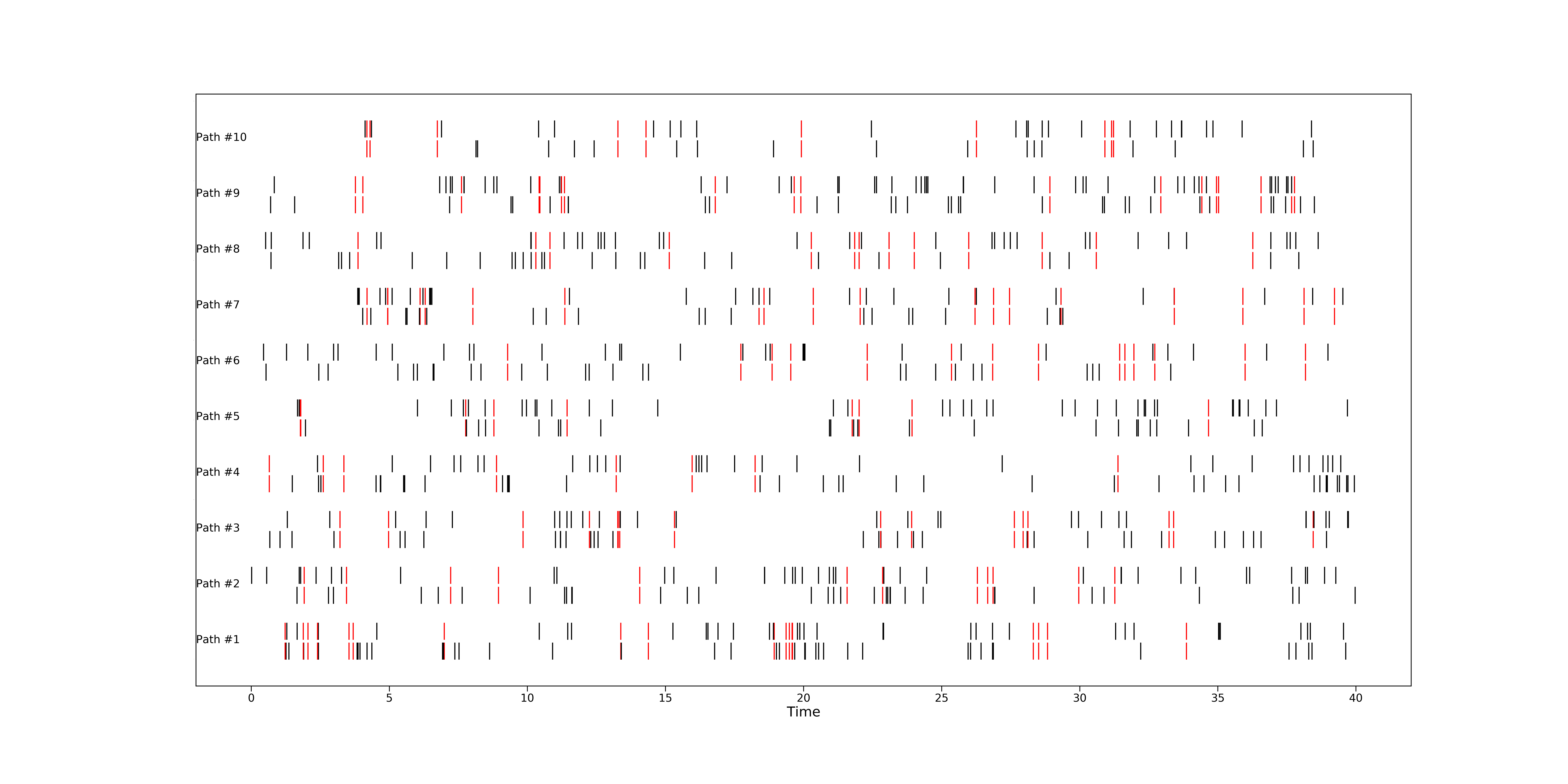}
	\caption{Bar plot of 10 paths of a bivariate point Hawkes process. Red bars represent common events, black bars represents idiosyncratic events.}
\end{figure}
\begin{figure}[h]
	\hspace*{-2cm}
	\includegraphics[width=1.2\textwidth]{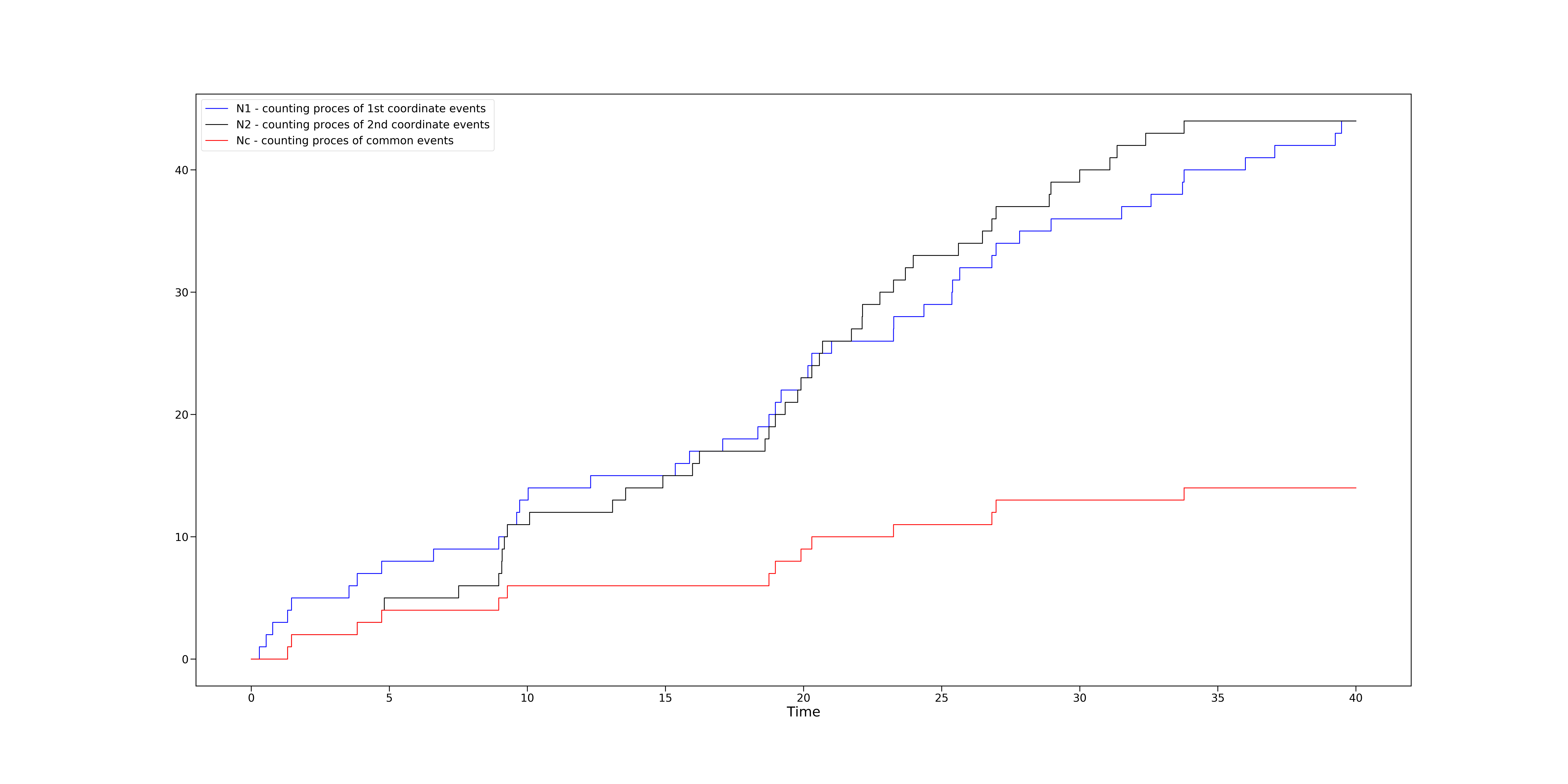}
	\caption{Plot of a single path of counting processes associated with 2-variate Hawkes process.}
\end{figure}

\end{example}

\begin{example}\label{2}$\ $ {\bf Bivariate Hawkes process}
		\medskip

Here we apply our pseudo-algorithm to Example \ref{commonevents} with $d=2$ and $E_1 = E_2 = \bR$.
We let:
\[
\eta_1(t,dy_1) = \alpha_1 \varphi_{\mu_1, \sigma_1} (y_1 ) d y_1,
\quad
\eta_2(t,dy_2) = \alpha_2 \varphi_{\mu_2, \sigma_2} (y_2 ) d y_2,
\]
\[
\eta_1(t,dy_1) = \alpha_c \varphi_{\mu_c, \sigma_c} (y_1 ) \varphi_{\mu_c, \sigma_c} (y_2 ) d y_1 dy_2
\]
where
$\alpha_i \geq 0$, $i \in \set{1,2,c}$, $\varphi_{\mu, \sigma}$ is the one dimensional Gaussian density function with mean $\mu$ and variance $\sigma^2$, and:

\[
	w_{1,i}(t,s) = w_{2,i}(t,s) = w_{c,i}(t,s) = e^{- \beta_i (t-s)}, \qquad i =1,2,c.
\]
Moreover, we set:
\begin{align*}
	 g_{1,1}(x_1) &=g_{1,1},& g_{2,1}(x_1) &= 0 ,& g_{c,1}(x_1) &= g_{c,1}, \\
	 g_{1,2}(x_2) &=0,& g_{2,2}(x_2) &= g_{2,2} ,&g_{c,2}(x_2) &= g_{c,2}, \\
	 g_{1,c}(x) &= 0, & g_{2,c}(x) &= 0, & g_{c,c}(x) &= g_{c,c}, \\
\end{align*}
and we take
\begin{align*}
\phi_1(x,dy_1) &= \1_{E_1 \times \Delta}(x) \varphi_{ a_1x_1, \sigma_1 }(y_1)dy_1
+\1_{\Delta  \times E_2 }(x)  \varphi_{ 0, \sigma_1 }(y_1)dy_1, \\
\phi_2(x,dy_2) &= \1_{\Delta \times E_2}(x)  \varphi_{ a_2x_2, \sigma_2 }(y_2)dy_2 +  \1_{E_1 \times \Delta }(x) \varphi_{ 0, \sigma_1 }(y_1)dy_2, \\
\phi_c(x,d y_1, dy_2) &= \1_{E_1 \times \Delta} (x)  \varphi_{a_c x_1, \sigma_c }(y_1) \varphi_{0, \sigma_c}(y_2)
+
\1_{\Delta \times E_2} (x)  \varphi_{0, \sigma_c} (y_1) \varphi_{a_c x_2, \sigma_c} (y_2) \\
&+
\1_{E_1 \times E_2} (x)  \varphi_{a_c x_1, \sigma_c} (y_1) \varphi_{a_c x_2, \sigma_c} (y_2),
\end{align*}
with the following values of the parameters:
\[
\begin{array}{|c|c|c|c|c|c|c|c|}
\hline
i & \alpha_i & \mu_i  & \sigma_i &  \beta_i  & a_i &  g_{i,i} & g_{c, i}\\ \hline \hline
1 & 0.4 & 2 & 0.16331 & 0.41175 & 0.9 & 0.3 & 0.1 \\ \hline
2 & 0.4 & -2 & 0.16331 & 0.41175 & 0.9 & 0.3 & 0.1 \\ \hline
c & 0.2 & 0 & 0.16331 & 0.81175 & 1.1 & 0.4 & 0.4\\ \hline
\end{array}
\]
A simulated sample path is presented on Figure 3.

\begin{figure}[h]
	\hspace*{-2cm}	
	\includegraphics[width=1.2\textwidth]{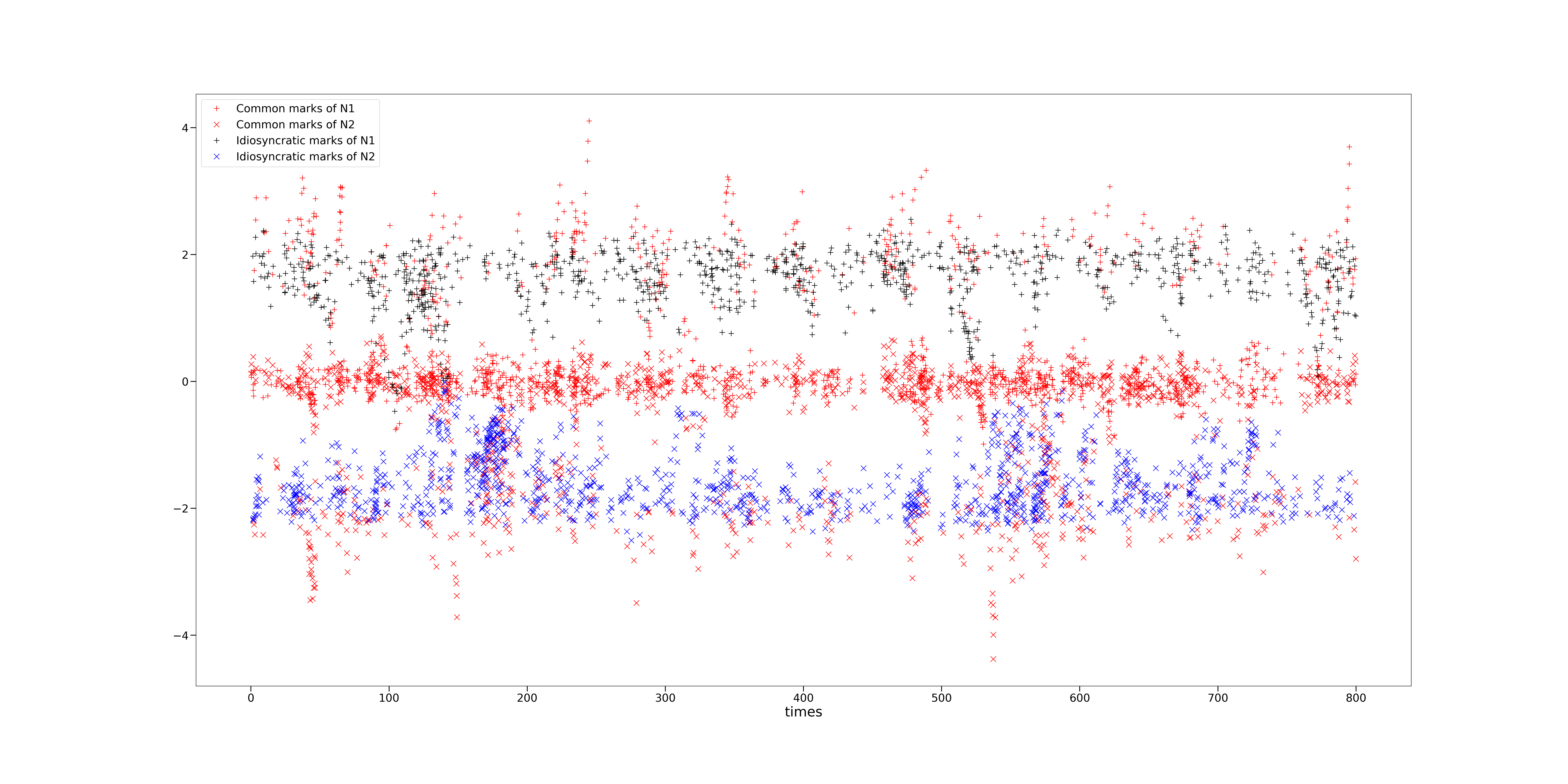}
	\caption{Plot of a simulated path of the bivariate Hawkes process specified in Example \ref{2}.}
\end{figure}
\end{example}

\section{Markovian aspects of a generalized  bivariate Hawkes process}\label{Ex:MarkovHawkes}
An important class	 of  Hawkes processes considered in the literature is the one of Hawkes processes for which the Hawkes kernel is given in terms of exponential decay functions. See, e.g., \cite{Cin2011}, \cite{Oak1975}, \cite{Zhu2013}.  
One interesting and useful aspect of such processes is that they can be extended to Markov processes, a feature that we term the {\textit{Markovian aspects} of  a generalized bivariate Hawkes process} .

	{To simplify the presentation, we will discuss Markovian aspects of generalized  bivariate Hawkes processes specified in Example \ref{1}. Using this specification we end up with the Hawkes kernel $\kappa$  of the form:
	\begin{align}\label{eq:kernel-common}
		\kappa(t,dy)
		=& \ \lambda^1_t \delta_{(1,\Delta)} (dy)
		+
		\lambda^2_t \delta_{(\Delta,1)} (dy)
		+
		\lambda^c_t \delta_{(1,1)} (dy),
	\end{align}
	where, for $i=1,2,c$, we have   $\lambda^i_0 := \eta_i(0)$ and
	\begin{align}
		\lambda^i_t = \  \alpha_i & + ( \lambda^i_0 - \alpha_i ) e^{-\beta_i t} \nonumber
		\\ \label{lambdeczka-i}
		&
		+
		\int_{(0,t) \times E^\Delta } e^{-\beta_i (t-u)}
		\Big(
		\vartheta_{i,1} \1_{ \set{1} \times \Delta}(x)
		\\ \nonumber
		& +  \vartheta_{i,2} \1_{ \Delta \times \set{1} }(x)
		+  \vartheta_{i,c} \1_{ \set{1} \times \set{1}}(x)  \Big) N(du,dx).
	\end{align}
We now refer to canonical space as in Section \ref{sec:existence}, and to the random measure $\nu$ corresponding to $\kappa$ as in \eqref{eq:nu-Hawkes-G}.
	So, using Theorem \ref{prop:main} we see that there exists a unique probability $\bP_\nu$ such that the canonical process $N$ given as in \eqref{eq:NN} is a generalized multivariate Hawkes process  with Hawkes kernel $\kappa$.		

	The coordinates of $N$ (cf. \eqref{N-i}) reduce here to counting (point) processes
	\begin{equation*}
	N^1_t  = N^1((0,t],\set{1})=N((0,t],\set{1} \times \set{1, \Delta}),
	\end{equation*}
	and
	\begin{equation*}
	N^2_t = N^2((0,t],\set{1})=N((0,t],\set{1, \Delta} \times \set{1}  ).
	\end{equation*}

It is straightforward to verify (upon appropriate integration of the kernel $\kappa$ i.e. over $\set{1} \times \set{1, \Delta}$ for $N^1$ and $\set{1, \Delta} \times \set{1}$ for $N^2$) that the $\bF^N$--intensity of process $N^i$, say $\wh \lambda^i$, is given as
		\begin{equation}\label{eq:hatlambda}
			\wh \lambda^i_t=\lambda^i_t+\lambda^c_t,\quad t\geq 0,
			\end{equation}
			$i=1,2$.	Let
\[
\bar{N}^c_t =  [N^1, N^2]_u, \qquad
\bar{N}^1_t =  N^1_u - \bar{N}^c_u, \qquad
\bar{N}^2_t =  N^2_u - \bar{N}^c_u,
\]
where $[N^1, N^2]$ is the square bracket of $N^1$, $N^2$. Then, for $i=1,2,c$, the equality \eqref{lambdeczka-i} can be written as
\begin{align}\label{eq:lambdeczka-i}
\lambda^i_t =  \alpha_i &+ ( \lambda^i_0 - \alpha_i ) e^{-\beta_i t}\\
\nonumber
& + \int_{(0,t)} e^{-\beta_i(t-u)}
\Big(\vartheta_{i,1}
d \bar{N}^1_u
+
\vartheta_{i,2} d \bar{N}^2_u
+
\vartheta_{i,c} d \bar{N}^c_u
\Big)
\end{align}
for $t\geq 0$. This follows from the fact that $[N^1, N^2]$ counts common jumps of $N^1$ and $N^2$, so for $i=1,2$ the process $\bar{N}^i$ is counts the idiosyncratic jumps of $N^i$, that is the jumps that do not occur simultaneously with the jumps of $N^j$, $j\ne i$. In particular, expression \eqref{eq:lambdeczka-i} allows us to give the interpretation of the parameters $\vartheta_{i,j}$, $i,j\in \{1,2,c\}$, namely the parameter $\vartheta_{i,j}$ describes the impact of the jump of the process $\bar{N}^j$ on the intensity of $\bar{N}^i$.

	Now, let us consider a bivariate counting process  $\wt{N} := (N^1,N^2)$. 	Note that we may, and we do, identify process $\wt{N}$ with our bivariate generalized Hawkes process $N$:
	\[
	T_0 =0, \quad
	T_n = \inf \set{ t > T_{n-1} : \Delta \wt N_t \neq (0,0) },
	\]
	and for $i=1,2$ \[ X^i_n =
	\begin{cases}
	1 & \text{ if } \Delta N^i_{T_n} = 1,\\
	\Delta & \text{ if } \Delta N^i_{T_n} =0.
	\end{cases}
	\]

	Also, note that we may, and we do, identify the process $\wt{N}$  with a random measure $\mu^{\wt N}$ on $\bR_+ \times \wt{E}$,  where $\wt{E} = \set{(1,0),(0,1), (1,1)}$, given by
	\[
	\mu^{\wt N}(dt,dy) =
	\sum_{n \geq 0} \delta_{(T_n, Y_n)}(dt, dy) \1_\set{T_n <\infty} ,\]
	where
	$Y^i_n = \1_\set{X^{i}_n = 1}$. Using \eqref{eq:kernel-common} we see that
	 $\bF^{\wt N}$--compensator of $\mu^{\wt N} $ is given by
	\begin{equation}\label{eq:nueczka}
	\wt{\nu}(dt, dy) = \1_{]\!]0, T_\infty [\![ } \wt{\kappa}(t, dy) dt,
	\end{equation}
	where
	\begin{align}\label{eq:kappeczka}
		\wt{\kappa}(t,dy)
		=& \ \lambda^1_t \delta_{(1,0)} (d y)
		+
		\lambda^2_t \delta_{(0,1)} (d y)
		+
		\lambda^c_t \delta_{(1,1)} (d y).			
	\end{align}
	Thus, we may slightly abuse terminology and call $\wt{N}$ a generalized bivariate Hawkes process.
	
	\begin{theorem}\label{thm:markov}
		Let $N$ be a Hawkes process defined as above. Then
		
		i) The process $\wt N = ( N^1_t, N^2_t)_{t \geq 0}$ is not a Markov process.
		
		ii) The process
		\[Z=(\lambda^1_t, \lambda^2_t, \lambda^c_t, N^1_t, N^2_t)_{t \geq 0}\]
		is a Markov process with the strong generator $A$ acting on $C^\infty_c(\bR_+^5)$ given by
		\begin{align}
		A & v(  \lambda^1, \lambda^2, \lambda^c, n^1, n^2) \label{eq:hawkes-generator} \\ \nonumber
		& =
		\beta_1( \alpha_1 - \lambda^1)  \frac{ \partial }{ \partial \lambda^1 } v(  \lambda^1, \lambda^2, \lambda^c, n^1, n^2)
		+
		\beta_2( \alpha_2 - \lambda^2)  \frac{ \partial }{ \partial \lambda^2 } v(  \lambda^1, \lambda^2, \lambda^c, n^1, n^2)
		\\ \nonumber
		&+
		\beta_c( \alpha_c - \lambda^c)  \frac{ \partial }{ \partial \lambda^c } v(  \lambda^1, \lambda^2, \lambda^c, n^1, n^2)
		\\\nonumber
		&+( v(\lambda^1 + \vartheta_{1,1}, \lambda^2 + \vartheta_{2,1}, \lambda^c + \vartheta_{c,1}, n^1+1, n^2) - v(\lambda^1, \lambda^2, \lambda^c, n^1, n^2) ) \lambda^1
		\\ \nonumber
		&+ ( v(\lambda^1 + \vartheta_{1,2}, \lambda^2 + \vartheta_{2,2}, \lambda^c + \vartheta_{c,2}, n^1, n^2+1) - v(\lambda^1, \lambda^2,\lambda^c, n^1, n^2) ) \lambda^2
		\\ \nonumber
		&+ ( v(\lambda^1 + \vartheta_{1,c}, \lambda^2 + \vartheta_{2,c}, \lambda^c + \vartheta_{c,c}, n^1+1, n^2+1) - v(\lambda^1, \lambda^2,\lambda^c, n^1, n^2) ) \lambda^c.
		\end{align}
				\end{theorem}

	\begin{proof}
	
	%
	i)
	From \eqref{eq:lambdeczka-i} and \eqref{eq:kappeczka} we see that for any $t>0$  the quantity $\widetilde \nu(dt,dy)$ given in \eqref{eq:nueczka}  depends on the entire path of $\wt N$ until time $t$. Thus, by Theorem 4 in \cite{HeWan1984}, the process $\wt{N}$ is not a Markov process.
	
	ii)
	First note that \eqref{eq:lambdeczka-i} can be written as
		\begin{align*}\
			\lambda^i_t -  \alpha_i & =
			 e^{-\beta_i t} \bigg( \lambda^i_0 - \alpha_i  + \int_{(0,t)} e^{\beta_i u}
			\Big(\vartheta_{i,1}
			d \bar{N}^1_u
			+
			\vartheta_{i,2} d \bar{N}^2_u
			+
			\vartheta_{i,c} d \bar{N}^c_u
			\Big)  \bigg).
		\end{align*}
	Hence using stochastic integration by parts one can show that $\lambda^i$ 
	can be represented as
	\[
	\lambda^i_t = \lambda^i_0 + \int_0^t \beta_i( \alpha_i - \lambda^i_u) du + \int_{(0,t)} \Big( \vartheta_{i,1}
	d \bar{N}^1_u +
	\vartheta_{i,2} d \bar{N}^2_u +
	\vartheta_{i,c} d \bar{N}^c_u \Big).
	\]
	This and \eqref{eq:kappeczka} implies that the  process $Z$ is an $\bF^Z$--semimartingale with characteristics (with respect to cut-off function $ h(x) = x \1_{|x|<1}$)
	\[
	B_t = \int_0^t b_u du,
	\quad
	C_t  =0_{4 \times 4},
	\]
	\[
	\nu(dt, dy_1, dy_2, dy_c, d z_1, d z_2)
	=
	\nu_t (dy_1, dy_2, dy_c, d z_1, d z_2)dt,
	\]
	where
	\[
	b_t
	:=\left(
	\beta_1( \alpha_1 - \lambda^1_{t-})   , \
	\beta_2( \alpha_2 - \lambda^2_{t-})   , \
	\beta_c( \alpha_c - \lambda^c_{t-})   , \
	0, \ 0
	\right)'
	\]
	and
	\begin{align}\label{eq:miarahawksa}
		&\nu_t(dy_1,dy_2, dy_c, d z_1, d z_2)
		\\ \nonumber
		&:=
		\lambda^1_{u-} \delta_{(\vartheta_{1,1},\vartheta_{2,1},\vartheta_{c,1}, 1,0)}(dy_1, dy_2, dy_c, d z_1, d z_2)
		\\ \nonumber
		& \quad +
		\lambda^2_{u-} \delta_{(\vartheta_{1,2},\vartheta_{2,2},\vartheta_{c,2}, 0,1)}(dy_1, dy_2, dy_c, d z_1, d z_2)
		\\ \nonumber
		& \quad +
		\lambda^c_{u-} \delta_{(\vartheta_{1,c},\vartheta_{2,c},\vartheta_{c,c}, 1,1)}(dy_1, dy_2, dy_c, d z_1, d z_2).
	\end{align}	
	This, by Theorem II.2.42 in \cite{js1987}, implies that for any function  $v\in C^2_b(\bR^5)$ the process $M^v$ given as
\begin{align*}
	M^v_t &= v(Z_t)
	- \int_0^t A v(Z_u) du
	\\
	&= v(\lambda^1_t,\lambda^2_t, \lambda^c_t, N^1_t, N^2_t)
	- \int_0^t A v(\lambda^1_u,\lambda^2_u, \lambda^c_u, N^1_u, N^2_u) du,\quad t \geq 0,
\end{align*}	
	is an $\bF^Z$--local martingale. Hence, for any $v\in C^\infty_c(\bR^5)$ the process defined above is a martingale under $\bP$, since
	$v$ and  $A v$ are bounded, which follows from the fact that $v\in C^\infty_c(\bR^5)$ has compact support, and thus the local martingale $M^v$ is a martingale for such $v$.
	Consequently, the process $Z$ solves martingale problem for $(A,\rho)$, where $\rho$ is the deterministic initial distribution~of~$Z$, that is $ \rho(dz) = \delta_{Z_0}(dz)$.

	We will now verify that $Z$ is a Markov process with generator $A$ given in \eqref{eq:hawkes-generator} using Theorem 4.4.1 in  \cite{ethkur1986}.

	For this, we first observe that  parameters determining $A$, i.e.
	\[
	\cI = \set{1, \ldots, 5}, \ \cJ = \varnothing , \ a=0, \ \alpha=0,\ c=0, \ \gamma=0, \ m=0,
	\]
	\[
	b= (\alpha_1 \beta_1, \alpha_2 \beta_2, \alpha_3 \beta_3, 0,0)', \quad \beta=\mathop{diag}(-\beta_1,-\beta_2,-\beta_c,0,0 ),
	\]
	\[
	\mu_1= \delta_{(\vartheta_{1,1},\vartheta_{2,1},\vartheta_{c,1}, 1,0)},
	\quad
	\mu_2 = \delta_{(\vartheta_{1,2},\vartheta_{2,2},\vartheta_{c,2}, 0,1)},  
	\]
	\[\mu_3= \delta_{(\vartheta_{1,c},\vartheta_{2,c},\vartheta_{c,c}, 1,1)},
	\quad
	 \mu_4= \mu_5=0,\] 
	are admissible in the sense of Definition 2.6 in  \cite{duffilsch2003}.
	
	Thus, invoking Theorem 2.7 in \cite{duffilsch2003} we conclude that there exists a unique regular affine semigroup $(P_t)_{t\geq 0 }$ with infinitesimal generator $A$ given by \eqref{eq:hawkes-generator}. Hence,  there exists a unique regular affine process with generator $A$ and with transition function $P$ defined by $(P_t)_{t\geq 0 }$.
	Since $A$ is a generator of regular affine process it satisfies the Hille-Yosida conditions (cf. Theorem 1.2.6 in \cite{ethkur1986}) relative to the Banach space $B(\bR^5)$ of real valued, bounded and measurable functions on $\bR^5$. Moreover, from Corollary 1.1.6 in \cite{ethkur1986} it follows that $A$ is a closed operator.
	Now, using Theorem 4.4.1 in  \cite{ethkur1986} we obtain  that $Z$ is a Markov process with generator $A$. Moreover, $P$ is the  transition function of $Z$.			
\end{proof}

Let us note that using analogous argument as in the proof of Theorem \ref{thm:markov} we can prove that the process $Y^1:=(\lambda^1_t+\lambda^c_t, N^1_t)_{t \geq 0}$ is a Markov process in filtration $\mathbb{F}^{Z}$ provided that 
parameters of $\lambda^k$, $k\in \set{1,c}$, satisfy
\begin{align*}
	\vartheta_{1,2} = \vartheta_{c,2}= 0,\quad
    \beta_1 = \beta_c,    \quad  \vartheta_{1,c} + \vartheta_{c,c} = \vartheta_{1,1} + \vartheta_{c,1}.
	\end{align*}	
	Analogous statement is valid for $Y^2:=(\lambda^2+\lambda^c, N^2)$.

\section{Applications}\label{sec:app}

The are numerous potential applications of the generalized  multivariate Hawkes processes. Here we present a brief description of possible applications in seismology, in epidemiology and in finance.

\subsection{Seismology}\label{sec:seismology}

In the Introduction to \cite{Ogata1999} the author writes:

\medskip
``Lists of earthquakes are published regularly by the seismological services of
most countries in which earthquakes occur with frequency. These lists supply at
least the epicenter of each shock, focal depth, origin time and instrumental
magnitude.

Such records from a self-contained seismic region reveal time series of extremely
complex structure. Large fluctuations in the numbers of shocks per time unit,
complicated sequences of shocks related to each other, dependence on activity in
other seismic regions, fluctuations of seismicity on a larger time scale, and changes
in the detection level of shocks, all appear to be characteristic features of such
records. In this manuscript the origin times are mainly considered to be modeled by
point processes, with other elements being largely ignored, except that the ETAS
model
 and its extensions use data of magnitudes and epicenters.''

\medskip
In particular, the dependence on (simultaneous) seismic activity in other seismic regions has been ignored in the classical univariate ETAS\footnote{The Epidemic-type Aftershock-sequences Model} model, and in all other models that we are aware of.

The ETAS model is a univariate self-exciting point process, in which the shock intensity at time $t$, corresponding to a specific seismic location,  is designed as (cf. Equation (17) in \cite{Ogata1999})
\begin{equation}\label{eq:Og19}
\lambda(t|H_t)=\mu +\sum_{t_m<t}\frac{K_m}{(t-t_m+c)^p}.
\end{equation}
In the above formula, $H_t$ stands for the history of after-shocks at the given location, $\mu$ represents the background occurrence rate of seismic activity at his location, $t_m$s are the times of occurrences of all after-shocks that took place prior to time $t$ at the specific seismic location, and
\[K_m=K_0e^{\alpha(M_m-M_0)},\]
where $M_m$ is the magnitude of the shock occurring at time $t_m$, and $M_0$ is the cut-off magnitude of the data set; we refer to \cite{Ogata1999} for details. As said above, dependence between (simultaneous) seismic activity in different seismic regions has been ignored in the classical univariate ETAS model.

 Below we suggest a possible method to  construct a generalized multivariate Hawkes process that may offer a good way of  modeling of joint seismic activities at various locations, accounting for dependencies between seismic activities at different locations and for consistencies with local data.

We will now briefly describe this construction that leads to a plausible model, which we name the {\it multivariate generalized ETAS model}. Towards this end we consider a GMHP $N$ (cf. Definition \ref{def:multiHawkesi}), where the index $i=1,\ldots,d$ represents the $i$-th seismic location, and where the set $E_i=\mathfrak{M}_i:=\set{m_1,m_2,\ldots,m_{n_i}}$ of marks is a discrete set whose elements represent possible magnitudes of seismic shocks with epicenter at location $i$. In the corresponding Hawkes kernel $\kappa$ the measure $\eta(t,dy)$ represents the time-$t$  background distribution of shocks' across all seismic regions, and the measure $f(t,s,dy,x)$ represents the feedback effect.

For the purpose of illustration, let $d=2$. Suppose that local seismic data are collected for each location to the effect of producing local kernels of the form

\begin{align}\label{kappa1-ETAS}
 \kappa^i(t,\set{y_i})
 &= \chi_i(t,\set{y_i}) 	\nonumber \\
 &+  \int_{(0,t)\times E_1} h_{i,1}(t,s,x_1,\set{y_i}) N^{idio,1}(ds,dx_1)  \\ \nonumber	&+  \int_{(0,t)\times E_2} h_{i,2}(t,s,x_2,\set{y_i})  N^{idio,2}(ds,dx_2)  \\ \nonumber		
 &+  \int_{(0,t)\times E_1\times E_2}  h_{i,c}(t,s,x,\set{y_i}) N(ds,dx), \quad i=1,2.
 \end{align}

In particular, the quantity $\lambda^i(t):=\kappa^i(t,E_i)=\sum_{y_i\in \mathfrak{M}_i}\kappa^i(t,y_i)$ represents the time-$t$ intensity of seismic activity at the $i$-th location.

In order to produce an ETAS type model, we postulate that
\[\sum_{y_i\in \mathfrak{M}_i}h_{i,j}(t,s,x_j,\{y_i\})=\frac{K_{i,j,0}e^{\alpha_{i,j}(x_j-x_{j,0})}}{(t-s+c)^{p_{i,j}}},\]
for $j=1,2$ and
\[\sum_{y_i\in \mathfrak{M}_i}h_{i,c}(t,s,x,\{y_i\})=\frac{K_{i,c,0}e^{\alpha_{i,c}[(x_1-x_{1,0})+(x_2-x_{2,0})]}}{(t-s+c)^{p_{i,c}}}.\]

Thus,
\begin{align}
\lambda^i(t)&=\sum_{y_i\in \mathfrak{M}_i} \Bigg (\chi_i(t,\{y_i\}) + \sum_{j=1}^2\sum_{t_{j,m}<t}\frac{K_{i,j,0}e^{\alpha_{i,j}(X_{j,t_{j,m}}-x_{j,0})}}{(t-s+c)^{p_{i,j}}} \nonumber \\
&\quad + \sum_{t_{c,m}<t}\frac{K_{i,c,0}e^{\alpha_{i,c}[(X_{1,t_{c,m}}-x_{1,0})+(X_{2,t_{c,m}}-x_{2,0})]}}{(t-s+c)^{p_{i,c}}}\Bigg ),
\end{align}
where
\begin{itemize}
\item $t_{j,m}$s are the times of occurrences of  after-shocks that took place prior to time $t$ only at the $i$-th seismic location, and $X_{j,t_{j,m}}$ is the magnitude of the after shock at location $i$ that took place at time $t_{j,m}$;
\item $t_{c,m}$s are the times of occurrences of  after-shocks that took place prior to time $t$ both seismic locations, and $X_{j,t_{c,m}}$ is the magnitude of the after shock at location $i$ that took place at time $t_{c,m}$.
\end{itemize}

The classical univariate ETAS model has been extended in  \cite{Ogata1998} to the (classical) univariate space-time ETAS model (see also Section 5 in \cite{Ogata1999}).
It is important to note that our  generalized multivariate Hawkes process may also be used as an useful generalization  of the space-time extension of the multivariate generalized ETAS model. In order to see this, let us consider the model (2.1) in \cite{Ogata1998} with $g$ as in Section 2.1 in \cite{Ogata1998}, that is (in the original notation of \cite{Ogata1998}, which should not be confused with our notation)

\begin{equation}\label{eq:timespaceseism}
\lambda(t,x,y|H_t)=\mu(x,y)+\int_0^t\int\int_{A}\int_{M_0}^\infty g(t-s,x-\xi,y-\eta;M)N(ds,d\xi,d\eta,dM).
\end{equation}

Then, coming back to our generalized multivariate Hawkes process, let the seismic location $i=1,2$ be identified with a point in the plane with coordinates $(a_i,b_i)\in \R^2$. Next, let the set of marks $E_i$ be given as
\begin{equation}\label{eq:spacemarks}
E_i:=D_i\times \mathfrak{M}_i,
\end{equation}
where $D_i=[a_i-a'_i,b_i-b'_i]\times [a_i+a''_i,b_i+b''_i]$ for some positive numbers $a'_i,a''_i,b'_i,b''_i.$ This will lead to a space-time generalized multivariate Hawkes process that will be studied elsewhere.

\subsection{Epidemiology}

It was already observed by Hawkes in \cite{Hawkes1971} that Hawkes processes may find applications in epidemiology for modeling spread of epidemic diseases accounting for various types of cases, such as children or adults, that can be taken as marks. This insight has been validated over the years in numerous studies. We refer for example to \cite{SHH2019,RMKCX2018,KELLY2019100354} and the references therein.

It is important to account for the temporal and spatial aspects in the modeling of spread and intensity of epidemic and pandemic diseases, such as COVID-19. We believe that
the variant of the generalized multivariate Hawkes process that we described at the end of Section \ref{sec:seismology} may offer a valuable tool in this regard. This will be investigated in a follow-up work.

\subsection{Finance}

Hawkes processes have found important applications in finance over the past two decades. We refer to \cite{Haw2017} for a relevant survey.  Here, we briefly discuss a possible application in finance of the generalized multivariate Hawkes processes.

In a series of papers \cite{BacDelHofMuz2013}, \cite{BacMuz2014}, \cite{BDHM} introduced a multidimensional model for stock prices driven by (multivariate) Hawkes processes. The  model for stock prices is formulated in \cite{BacDelHofMuz2013} via a marked point process $N= (T_n,Z_n)_{ n \geq 1 }$, where $Z_n$  is a random variable taking values in $\set{1, \ldots, 2d}$, and the compensator $\nu$ of $N$ has the form (it is assumed that $T_\infty = \lim_{n\rightarrow \infty}T_n=\infty$)
\[
\nu(dt,dy) = \sum_{i=1}^{2d} \delta_{i}(dy) \lambda_i(t) dt,
\]
where
\begin{equation*}
\lambda_i(t) = \mu_i +  \sum_{j=1}^{2d} \int_{(0,t)}  \phi_{i,j}(t-s) N(ds \times \set{j} ), \quad t\geq 0,
\end{equation*}
with $\mu_i \in \bR_+$ and  functions $\phi_{i,j}$ from $\bR_+$ to $\bR_+$.
Let us define the processes $N^i,$ $i=1,\ldots 2d$, by
\[
N^i((0,t] ) = \sum_{n \geq 1 } \1_{\set{T_n \leq t} \cap \set{Z_n =i} } , \quad t\geq 0.
\]
Note that the above implies that $N^1, \ldots, N^{2d}$ have no common jumps and
the $\bF^N$-intensity  of $N^i$ is given by $\lambda_i$  and  can be written in the form
\begin{equation*}
\lambda_i(t) = \mu_i +  \sum_{j=1}^{2d} \int_{(0,t)}  \phi_{i,j}(t-s) N^j(ds), \quad t\geq 0.
\end{equation*}

In \cite{BacDelHofMuz2013} it is assumed that a $d$-dimensional vector of assets prices  $S=(S^1, \ldots, S^d)$ is based on $N$ via representation
\[
S^i_t  = N^{2i-1}((0,t]) - N^{2i}((0,t]), \quad t \geq 0, \, i=1, \ldots, d.
\]
The obvious interpretation is that $N^{2i-1}$ corresponds  to an upward jump of the i-th asset whereas $N^{2i}$ corresponds  to an downward jump of i-th asset.
Bacry et.al. \cite{BacDelHofMuz2013} showed that within such framework some stylised facts about high frequency data,  such as microstructure noise and the Epps effect, are  reproduced.

Using the  GMHPs we can easily generalize their model in several directions. In particular, a model of stock price movements driven by a generalized multivariate Hawkes process  $N$ allows for common jumps in upward and/or downward direction. This can be done by setting the multivariate mark space of $N$ to be
\[
E^\Delta = \set{ e=  ( e_1, \ldots, e_{2d}) : e_i \in \set{1, \Delta} } \setminus \set{(\Delta, \ldots, \Delta )},
\]
and the $\mathbb{F}^N$-compensator of $N$ to be
\[
\nu(dt,dy) = \1_{]\!] 0, T_\infty[\![}( t)\sum_{e \in E^\Delta } \delta_{e}(dy) \lambda_e(t) dt,
\]
where
\[
\lambda_e(t) = \mu_e +  \int_{E^\Delta \times (0,t)}  \phi_{e,x}(t-s) N(ds \times dx), \quad e \in E^\Delta,\quad t\geq 0,
\]	
and where $\mu_e \in \bR_+$ and $\phi_{e,x}$ is a function from $\bR_+$ to $\bR_+$.

Including possibility of embedding co-jumps of the prices of various stocks in the book in the common excitation mechanism, may turn out to be important in modeling the book evolution in general, and in pricing basket options in particular.

\section{Appendix}
	In this appendix  we provide some auxiliary concepts and results that are needed in the rest of the paper.

\subsection{Conditional Poisson random measure: definition and specific construction}
	Let $(\Omega,\cF,\bP)$  be  a probability space and $(\mathcal{X},\mathbfcal{X} )$ be a Borel space. For a given sigma field $\cG \subseteq \cF$, we define  a $\cG$-conditionally Poisson random measure on $(\bR_+ \times \mathcal{X}, \cB(\bR_+) \otimes \mathbfcal{X})$ as follows:
	\begin{definition}
		Let $\nu$ be a $\sigma$-finite random measure on $(\bR_+ \times \mathcal{X}, \cB(\bR_+) \otimes \mathbfcal{X})$.
		A random measure $N$ on $(\bR_+ \times \mathcal{X}, \cB(\bR_+) \otimes \mathbfcal{X})$ is  a $\cG$-conditionally Poisson random measure with intensity measure $\nu$ if the following two properties are satisfied:
		\begin{enumerate}
			\item For every $C \in \cB(\bR_+) \otimes \mathbfcal{X}$ such that $\nu(C) < \infty$, we have
			\[
			\bP( N(C) = k| \cG)
			= e^{-\nu(C)}
			\frac{(\nu(C))^k}{k!}.
			\]
			\item For arbitrary $n=1,2,\ldots,$ and arbitrary disjoint sets $C_1, \ldots, C_n $ from $ \cB(\bR_+) \otimes \mathbfcal{X}$, such that $\nu(C_m) < \infty$, $m=1,\ldots,n$, the random variables
			\[
			N(C_1), \ldots, N(C_n)
			\]
			are $\cG$-conditionally independent.
		\end{enumerate}
	\end{definition}

	Clearly $\nu$ is $\cG$-measurable. 	Note that if $\cG$ is trivial $\sigma$-field (or if $N$ is independent of $\cG$), then $N$ is a Poisson random measure (see Chapter 4.19 in \cite{Sato2013}), which sometimes referred to as the Poisson process on $\bR_+ \times \mathcal{X}$ (see e.g. \cite{Kal2002})}. In this case $\nu$ is a deterministic $\sigma$-finite measure. For $\cG=\sigma(\nu)$, the $\sigma(\nu)$-conditional Poisson random measure is also known in the literature as Cox process directed by $\nu$ (see \cite{Kal2002}).

 Now we will provide a construction of a $\cG$-conditional Poisson random measure with the intensity measure given in terms of a specific kernel $g$. In fact, the measure constructed below is supported on sets from $\cB((0,T]) \otimes \mathbfcal{X}$, in the sense that for any set $C$ that has an empty intersection with $(0,T]\times \mathcal{X}$ the value of the measure is $0$ almost surely.

We begin by letting $
	g(t,y,dx)$ be a finite kernel from $(\bR_+ \times \mathcal{Y}, \cB(\bR_+) \otimes \mathbfcal{Y})$ to $(\cX, \mathbfcal{X})$, where $(\cY, \mathbfcal{Y})$ and $(\cX, \mathbfcal{X})$ are Borel  spaces,
	satisfying
	\begin{equation}\label{eq:gass}
		g(t,y,\mathcal{X}) = 0\quad \textrm{for}\quad  t > T.
	\end{equation}
	 Next, let $\partial$ be an element external to $\mathcal{X}$, and define kernel $g^\partial$ from $(\bR_+ \times \mathcal{Y}, \cB(\bR_+) \otimes \mathbfcal{Y})$ to $(\cX^\partial, \mathbfcal{X}^\partial)$ as
	\[
	g^\partial(t,y,dx) = \lambda(t,y)\gamma(t,y,dx),
	\] 	
	where
	\[
	\lambda(t,y)= g(t,y, \mathcal{X}), \quad
	\gamma(t,y,dx)=\frac{g(t,y, dx)}{g(t,y, \mathcal{X})} \1_\set{g(t,y, \mathcal{X}) > 0} + \delta_\partial (dx) \1_\set{g(t,y, \mathcal{X}) = 0}.
	\]	
	Suppose that
	\[
	\sup_{t \in [\ell(y),T] } \lambda(t,y) \leq \widehat{\lambda}(y)  <\infty,
	\qquad
	\gamma(t,y,A) = \int_{{(0,1]}} \1_A(\Gamma(t,y,u)) du, \quad A \in \mathbfcal{X},
	\]
for some measurable mappings $\ell : \cY \rightarrow [0,T] \cup \set{ \infty }$, $\widehat{\lambda}: \cY \rightarrow (0, \infty)$ and $\Gamma: \bR_+ \times \cY \times {(0,1]} \rightarrow \mathcal{X} $. Existence of such mapping $\Gamma$ is asserted by Lemma 3.22 in \cite{Kal2002}. In addition, let $D:[0, \infty) \times {(0,1]} \rightarrow \bN$ be as in Step 1 of our construction done in Section \ref{sec:constr}.
	
	Next, take   $Y$ to be a  $(\cY,\mathbfcal{Y})$-valued random element, which is  $\cG$-measurable, and let
	$Z$ and $(U_m, V_m, W_m)_{m=1}^\infty$ be independent random variables uniformly distributed on ${(0,1]}$  and independent of $\cG$. We now define a random measure $N$ on $(\bR_+ \times \mathcal{X}, \cB(\bR_+) \otimes \mathbfcal{X})$ as
	\begin{align}\label{eq:Ngen}
	N (dt,dx) &= \sum_{m =1}^{\infty } \delta_{(T_m, X_m)} (dt,dx) \1_\set{\ell(Y) < T, \, i \leq P, \, A_m \leq \lambda(T_m, Y) }
	\\ \nonumber
	&=\sum_{m =1}^P \delta_{(T_m, X_m)} (dt,dx) \1_\set{\ell(Y) < T, \, A_m \leq \lambda(T_m, Y) },
	\end{align}
	where $P$, $(T_m,A_m, X_m)_{m=1}^{\infty}$
	are  random variables defined by transformation of the sequence
	$Z$,$(U_m, V_m, W_m)_{m=1}^\infty$ and the random element $Y$  in the following way:

	\begin{align}\label{eq:Pkp1j-new}
	P &= D\big((T - \ell(Y)) \wh \lambda(Y) \1_\set{ \ell(Y) < T}, Z \big)
	\\
	\nonumber
	& = D\big((T - \ell(Y)) \wh \lambda(Y) ,Z\big) \1_\set{ \ell(Y) < T},
	\\
	\nonumber
	T_m &= \big(\ell(Y) + (T-\ell(Y))U_m \big)\1_\set{ \ell(Y) < T} + {\infty \1_\set{ \ell(Y) \geq T }},
	\\
	\nonumber
	A_m &=\widehat{\lambda}(Y) V_m \1_\set{ \ell(Y) < T}, 
	\\
	\nonumber
	X_m &= \left\{
          \begin{array}{ll}
            \Gamma(T_m, Y, W_m), & \hbox{if $\ell(Y) < T$,} \\
            \partial, & \hbox{if {$\ell(Y) \geq T$}.}
          \end{array}
        \right.
 	\end{align}

Using the above set-up we see that, for each $m=1,2,\ldots,$
\begin{align}
\nonumber
&\bP((T_m, A_m, X_m)\in dt\times da \times dx | \cG)\\
\label{eq:distTAXi}
&=	
\1_\set{ \ell(Y) < T }\frac{1}{(T - \ell(Y))\wh \lambda(Y) }\1_{(\ell (Y),T]\times (0,\wh \lambda(Y) ]}(t,a) \gamma(t,Y,dx) dt da
\\ \nonumber
& \quad +\1_\set{ \ell(Y) \geq T} \delta_{(\infty,0,\partial)}(dt,da,dx),
\end{align}
where $\delta_{(\infty,0,\partial)}$ is a Dirac measure.

Note that even though the random elements $X_m$, $m=1,2,\ldots,$ may take value $\partial$, the measure $N$ given in \eqref{eq:Ngen} is a random measure on $(\bR_+ \times \mathcal{X}, \cB(\bR_+) \otimes \mathbfcal{X})$ having support belonging to $\cB([0,T]) \otimes \mathbfcal{X}$.

Given the above, we now have the following result.
	\begin{lemma}\label{lem:Gcprm}
		The random measure $N$ defined by \eqref{eq:Ngen} is a $\cG$-conditionally Poisson random measure with intensity measure $\nu$ given by
		\begin{equation}\label{eq:nu-int}
		{\nu(C) = \int_C g(v, Y, dx) \1_{ {(} \ell(Y), {\infty  )}} (v) dv}, \qquad C \in \cB(\bR_+) \otimes \mathbfcal{X}.
		\end{equation}
 	\end{lemma}
	\begin{proof}
		To prove the result we consider $N((s,t] \times B)$ for {fixed } $0 \leq s \leq t  $, $B \in \mathbfcal{X}$. We have
		\begin{align*}
		N((s,t] \times B) &=
		\sum_{m =1}^{\infty }
		\delta_{(T_m, X_m)} ((s,t] \times B)
		\1_\set{ \ell(Y) < T, i \leq P , A_m \leq \lambda(T_m, Y) }
		\\
		&=
		\sum_{m =1}^{P } \1_\set{\ell(Y) < T, s < T_m \leq t, X_m \in B , A_m \leq \lambda(T_m, Y) }.
		\end{align*}
		First we will prove that, conditionally on $\cG$, the random variable $N((s,t] \times B)$ has the Poisson distribution with mean $\nu((s,t] \times B)$.
		Towards this end we observe that $P $ has, conditionally on $\cG$, the Poisson distribution with mean $(T - \ell(Y))\wh{\lambda}(Y) \1_\set{\ell(Y) < T}$ (see \eqref{eq:Pkp1j-new}), so
		\[
		\bP( P =k | \cG )
		= e^{-(T - \ell(Y))\wh{\lambda}(Y)\1_\set{\ell(Y) < T}} \frac{\Big((T - \ell(Y))\wh{\lambda}(Y)\1_\set{\ell(Y) < T}\Big)^k}{k!}, \quad k=0,1,\ldots\	.
		\]
		Moreover,  we conclude from \eqref{eq:distTAXi} that for $m=1,2,\ldots,$
		\begin{align}
		\nonumber
		&\bP(
		\ell(Y) < T ,
		s < T_m \leq t, X_m \in B , A_m \leq \lambda(T_m, Y)
		| \cG)
		\\ \nonumber
		\quad
		&
		= \1_\set{\ell(Y) < T}\int_s^t
		\bigg( \!\!\! \int_B \bigg( \int_0^{\lambda(u, Y )}
		\!\!\!\!\!\!\!\!\!\!\!\!
		\frac{1}{(T - \ell(Y))\wh{\lambda}(Y) }\1_{(\ell(Y),T]\times [0,\wh{\lambda}(Y)]}(u,a)  da \bigg)\gamma(u,Y,dx) \bigg)du
		\\ \nonumber
		&= \1_\set{\ell(Y) < T} \frac{1}{(T - \ell(Y))\wh{\lambda}(Y) } \int_s^t
		\1_{(\ell(Y),T]}(u)   \lambda(u, Y   )\gamma(u,Y,B) \, du
		\\ \label{eq:pstwo}
		&=
		{\1_\set{\ell(Y) < T}
		\frac{1}{(T - \ell(Y))\wh{\lambda}(Y) }
		\int_s^t
		\1_{(\ell(Y), {\infty )}}(u)   g(u,Y,B) \, du}=:p(Y),
		\end{align}
				where the last equality follows from \eqref{eq:gass}.

		Note that for $u\in \bR$ and $m=1,2,\ldots,$ we have\, \footnote{In the ensuing two formulae $i=\sqrt -1$.}
\[\bE (e^{iu\1_{\{\ell(Y) < T ,
		s < T_m \leq t, X_m \in B , A_m \leq \lambda(T_m, Y)
		\}}}| \cG)=(1-p(Y))+p(Y)e^{iu}. \]
This and the $\cG$-conditional independence of  $P $ and $(T_m,A_m, X_m)_{m=1}^\infty$ imply that
		\begin{align*}
		\bE( e^{ iu
			N((s,t] \times B)}
		| \cG)
		&={
e^{(e^{iu} -1) p(Y)(T - \ell(Y))\wh{\lambda}(Y)\1_\set{\ell(Y) < T}  }}	=
		e^{(e^{iu} -1) \int_s^t
			\1_{(\ell(Y), {\infty )}}(v)   g(v,Y,B) dv}	
		\\ &=
		e^{(e^{iu} -1) \nu((s,t] \times B)}	.
		\end{align*}
		Thus, the random variable $N((s,t] \times B)$ has the $\cG$-conditional Poisson distribution with mean equal to $ \nu((s,t] \times B)$.

Using standard monotone class  arguments we obtain
		that for arbitrary $C \in \cB(\bR_+) \otimes \mathbfcal{X} $ random variable   $N(C)$ has, conditionally on $\cG$, the Poisson distribution with mean $ \nu(C)$.

		Next, we will show that  for $0 \leq s_1 < t_1
		\leq s_2 < t_2 \leq \ldots  \leq s_n < t_n  $ and for sets $B_1, \ldots, B_n \in \mathbfcal{X}$ the random variables
		\begin{equation}\label{eq:N1Nn}
		N((s_1,t_1] \times B_1)
		,
		\ldots
		,
		N((s_n,t_n] \times B_n)
		\end{equation}
		are conditionally independent given $\cG$. 	Towards this end let us define
		\[
		S_r((s,t] \times B) :=\sum_{m=1}^r I_m((s,t] \times  B),
		\]
		for $r \in \bN$, $0 \leq s
		< t $, $B \in \mathbfcal{X}$,
		where
		\[
		I_m((s,t] \times  B) : =  \1_\set{s < T_m \leq t, \, X_m \in B , \, A_m \leq \lambda(T_m, Y ) }.
		\]
		Note  that the random variable $N((s,t] \times B)$ can be represented as \[ N((s,t] \times B) =   S_{P}((s,t] \times B). \]
		Using this representation we obtain  that
		\begin{align*}
		&J:=\bP( N((s_1,t_1] \times B_1) = l_1
		,
		\ldots
		,
		N((s_n,t_n] \times B_n) = l_n | \cG)
		\\
		&
		=\sum_{r=0}^\infty
		\bP\bigg(
		\bigcap_{j=1}^n S_r((s_j,t_j] \times B_j) =l_j
		,P = r
		\, \Big| \cG
		\bigg)
		\\
		&
=\sum_{r=l}^\infty
\bP\bigg(
\bigcap_{j=1}^n S_r((s_j,t_j] \times B_j) =l_j,
S_r\Big(\bR_+  \times \mathbfcal{X} \setminus \bigcup_{j=1}^n (s_j,t_j] \times B_j \Big) =r - l
\,\Big| \cG
\bigg) 		\bP(
P = r
| \cG
)  ,
		\end{align*}
		where $l = \sum_{j=1}^n l_j$. Now, from \eqref{eq:pstwo},	we see that the random vector
		\[
		\bigg(
		S_r((s_1,t_1] \times B_1)
		,
		\ldots
		,
		S_r((s_n,t_n] \times B_n) , S_r\Big(\bR_+ \times \mathbfcal{X} \setminus \bigcup_{j=1}^n (s_j,t_j] \times B_j \Big) \bigg)
		\]
		has, conditionally on $\cG$, the multinomial distribution with parameters $p_1, \ldots , p_{n+1}$ given by:
		\begin{align}
		\label{eq:pl}
		p_j &= p_j(Y) :=\bP(\ell(Y) < T , s_j < T_1 \leq t_j, X_1 \in B_j , A_1 \leq \lambda(T_1,  Y ) | \cG),
		\end{align}
		for $ j=1, \ldots,n $,
		and
		\[ p_{n+1} = 1- p_1 - \ldots -p_{n}.
		\]
		Hence, using the fact that $l = \sum_{j=1}^n l_j$,
		we deduce that
		\allowdisplaybreaks
		\begin{align*}
		J&
		=\sum_{r=l}^\infty
		\frac{r!}{l_1! \ldots l_n! (r-l)!}
		p_1^{l_{1}}
		\cdots
		p_{n}^{l_n}
		p_{n+1}^{r - l}
		\bP(
		P = r
		| \cG
		)
		\\
		&
		=
		\frac{1}{l_1! \ldots l_n!}
		p_1^{l_{1}}
		\cdots
		p_{n}^{l_n}
		\sum_{r=0}^\infty
		\frac{(r+l)!}{r!}
		p_{n+1}^{r}
		\bP(
		P = r+l
		| \cG
		)
		\\
		&=
		\frac{1}{l_1! \ldots l_n!}
		p_1^{l_{1}}
		\cdots
		p_{n}^{l_n}
		\sum_{r=0}^\infty
		\frac{(r+l)!}{r!}
		p_{n+1}^{r}
		e^{-(T - \ell(Y))\wh{\lambda}(Y)\1_\set{ \ell(Y) < T  }} \frac{\big((T - \ell(Y))\wh{\lambda}(Y) \1_\set{ \ell(Y) < T  } \big)^{r+l}}{(r+l)!}	
		\\
		&=\Bigg[
		\prod_{j=1}^n
		\frac{\Big(p_j (T - \ell(Y))\wh{\lambda}(Y) \1_\set{ \ell(Y) < T }\Big)^{l_j} }{l_j!}
		\Bigg]
		e^{(p_{n+1}-1)(T - \ell(Y))\wh{\lambda}(Y) \1_\set{ \ell(Y) < T  }}
		\\
		&=
		\prod_{j=1}^n
		\frac{\Big(p_j (T - \ell(Y))\wh{\lambda}(Y)\1_\set{ \ell(Y) < T }\Big)^{l_j} }{l_j!}
		e^{-p_j(T - \ell(Y))\wh{\lambda}(Y) \1_\set{ \ell(Y) < T }}
		\\
		&=\prod_{j=1}^n \bP(
		N((s_j,t_j] \times B_j) = l_j | \cG),
		\end{align*}
		where the last equality follows from the fact that
		$N((s_j,t_j] \times B_j)$ has the $\cG$-conditional Poisson distribution with mean equal to $ \nu((s_j,t_j] \times B_j) = p_j (T - \ell(Y))\wh{\lambda}(Y)\1_\set{ \ell(Y) < T }$, which is a consequence of  \eqref{eq:nu-int}, \eqref{eq:pstwo} and  \eqref{eq:pl}.
		Using standard use the monotone class arguments  we conclude  from \eqref{eq:N1Nn}
		that for arbitrary disjoint sets  $C_1, \ldots, C_n \in \cB(R_+) \otimes \mathbfcal{X} $ that random variables   $N(C_1), \ldots, N(C_n)$  are $\cG$-conditionally independent.
			The proof is now complete.
	\end{proof}

\subsection{Relation between conditional Poisson random measures  and doubly stochastic marked Poisson processes}

	We begin by recalling (cf. Chapter 6 in \cite{LasBra1995}) the concept of a doubly stochastic marked Poisson process. For this, we consider a filtration $\bF$ on $(\Omega,\cF,\bP)$.
	A marked point process $N$ on {$(\bR_+ \times \mathcal{X}, \cB(\bR_+) \otimes \mathbfcal{X})$} is an $\bF$-doubly stochastic marked Poisson process if
	there exist an $\mathcal{F}_0$-measurable random measure $\nu$ on 	{on $(\bR_+ \times \mathcal{X}, \cB(\bR_+) \otimes \mathbfcal{X})$} such that
	\begin{equation}\label{eq:DSMPP}
	\bP( N((s,t] \times B ) = k | \cF_s ) = e^{\nu((s,t] \times B)} \frac{(\nu((s,t] \times B))^k}{k!},
	\qquad
	0 \leq s < t, \
	B \in \cX.
	\end{equation}
Thus, for  $0 \leq s < t$, $B \in \mathbfcal{X}$ we have
	\begin{equation}\label{eq:F0intker}
	\nu((s,t] \times B) = \bE( N((s,t] \times B) | \cF_0).
	\end{equation}
	Hence, by analogy with the concept of the intensity of a Poisson random measure, the measure $\nu$ is called the $\cF_0$-intensity kernel of $N$ (see Chapter 6 in \cite{LasBra1995}).

Let now $\tilde N$  be marked point process on $(\bR_+ \times \mathcal{X}, \cB(\bR_+) \otimes \mathbfcal{X})$, such that its $\bF$-compensator $\tilde \nu$ is the $\cF_0$-intensity kernel in a sense that the property analogous to \eqref{eq:F0intker} holds,
  \[\tilde \nu((s,t] \times B) = \bE(\tilde  N((s,t] \times B) | \cF_0),\ 0 \leq s < t,\ B \in \mathbfcal{X}. \]
 Then, one can show (see Theorem 6.1.4 in \cite{LasBra1995}) that  $\tilde N$ is an $\bF$-doubly stochastic marked Poisson process,  i.e. the analog of \eqref{eq:DSMPP} holds with $\tilde N$ and $\tilde \nu$. The opposite statement is true as well (see Theorem 6.1.4 in \cite{LasBra1995}):  if  $\tilde N$ is an $\bF$-doubly stochastic marked Poisson process, then the $\bF$-compensator $\tilde \nu$ of $\tilde N$   is  an $\cF_0$-intensity kernel of $\tilde N$.

{Conditional Poisson random measures on  $(\bR_+ \times \mathcal{X}, \cB(\bR_+) \otimes \mathbfcal{X})$ are closely related to $\bF$-doubly stochastic marked Poisson processes. {It can be} shown  that if $N$ is an $\bF$-doubly stochastic marked Poisson process with intensity kernel $\nu$, then $N$ considered as a random measure is an  $\cF_0$-conditionally Poisson random measure with intensity kernel $\nu$.

This implies that for sets
		$B_1, \ldots, B_n \in \mathbfcal{X}$ and for $0 \leq s_1 < t_1 \leq s_2 < t_2 \leq \ldots \leq s_n < t_n \leq t$, $n \in \bN $, we have
\begin{align}\label{eq:poissonrm}
&
\bP
\Big(
\bigcap_{i=1}^n
\set{
	N((s_i, t_i] \times B_i )=l_i } | \cF_0 \Big) =
\prod_{i=1}^{n}
e^{\nu((s_i,t_i] \times B_i)} \frac{(\nu((s_i,t_i] \times B_i))^{l_i}}{l_i!}
\\ \nonumber
&
=
\prod_{i=1}^n
\bP
\big(\set{
	N((s_i, t_i] \times B_i )=l_i } | \cF_0 \big) .	
\end{align}

The next result, in a sense, complements our discussion of conditional Poisson random measures and doubly stochastic marked Poisson processes.
	\begin{proposition}\label{prop:fdsmpp2}
		i) 	Let $M$ be a marked point process {on $(\bR_+ \times \mathcal{X}, \cB(\bR_+) \otimes \mathbfcal{X})$}, which is a $\cG$-conditional Poisson random measure with intensity measure $\nu$, and let $\wh{\bF}^M$ be a filtration defined by family of $\sigma$-fields
		\[
		\wh{\cF}^M_t = \cG \vee \cF^{M}_t, \qquad t \geq 0.
		\]
		Then $M$ is an  $\wh{\bF}^M$-doubly stochastic marked Poisson process with $\wh{\cF}^M_0$-intensity kernel $\nu$ {being also $\wh{\bF}^M$-compensator of $M$}.
		
		\noindent ii)
		Let $\cN = (N_j)_{j \geq 1}$ be a family of  marked point processes {on $(\bR_+ \times \mathcal{X}, \cB(\bR_+) \otimes \mathbfcal{X})$}, which are $\cG$-conditional Poisson random measures (each $N_j$ with intensity measure $\nu_j$), and let $\wh{\bF}^\cN$ be a filtration defined by the family of $\sigma$-fields
		\[
		\wh{\cF}^\cN_t = \cG \vee \bigvee_{k \geq 1} \cF^{N_k}_t, \qquad t \geq 0.
		\]
		Suppose that $(N_j)_{j \geq 1}$ are $\cG$-conditionally independent. Then each $N_j$ is an  $\wh{\bF}^\cN$-doubly stochastic marked Poisson process with $\wh{\cF}^\cN_0$-intensity kernel $\nu_j$ { being also $\wh{\bF}^\cN$-compensator of $N_j$}.
	\end{proposition}

In the proof of Proposition \ref{prop:fdsmpp2} we will use the following elementary result, whose derivation is omitted:
	\begin{lemma}\label{lem:cond-indep}
		Let $\cG$ be a sigma field and let $A \in \cG$. Then for arbitrary measurable sets $B$ and $C$ which are conditionally independent given $\cG$ we have
		\[
		\bE( \1_B \1_{A \cap C}) = \bE( \bP (B|\cG) \1_{A \cap C}).
		\]
	\end{lemma}
	\begin{proof}(of Proposition \ref{prop:fdsmpp2})
		We will prove ii), the proof of i) is similar in spirit to the proof of ii) and in  fact a bit simpler.	
		Fix arbitrary $j \geq 1$. By assumption $N_j$ is a $\cG$-conditional Poisson random measure, so we have {for fixed $0 \leq s < t$ and $D \in \mathbfcal{X}$ }
		\begin{equation}\label{eq:NjgivenG}
				\bP( N_j ((s,t] \times D) = i | \cG) =
				e^{-\nu_j((s,t] \times D)} \frac{(\nu_j((s,t] \times D))^i}{i!},
		\end{equation}		
		where $\nu_j$ is $\cG = \wh{\cF}^{\cN}_0$-measurable random measure.
		In view of the definition of  $\wh{\bF}^\cN$-doubly stochastic marked Poisson process, of the above formula {and of  Proposition 6.1.4 in \cite{LasBra1995}} it suffices to show that for arbitrary set $F \in \wh{\cF}^\cN_s $ it holds
		\begin{equation}\label{eq:DSMPP1}
		\bE( \1_\set{N_j ((s,t] \times D)=k} \1_{F}) =
		\bE( \bP( N_j ((s,t] \times D) = k | \cG )\1_{F}) .
		\end{equation}
		Indeed \eqref{eq:DSMPP1} and \eqref{eq:NjgivenG} imply
		\[
						\bP( N_j ((s,t] \times D) = i | \wh{\cF}^\cN_s) =
		e^{-\nu_j((s,t] \times D)} \frac{(\nu_j((s,t] \times D))^i}{i!}.
		\]
		for $\wh{\cF}^{\cN}_0$-measurable random measure $\nu_j$.
		So that $N_j$ is a $\wh{\bF}^\cN$-doubly stochastic marked Poisson process  with $\wh{\cF}^{\cN}_0$-intensity kernel $\nu_j$. Then Proposition 6.1.4 in \cite{LasBra1995}  implies that $\nu_j$ is $\wh{\bF}^\cN$-compensator of $N^j$.

		To prove 	\eqref{eq:DSMPP1} we will use the Monotone Class Theorem.
		First note that sets $F$ for which \eqref{eq:DSMPP1} holds constitute  $\lambda$-system.
		Thus it suffices to show the above  equality for a $\pi$-system of sets which generates $\wh{\cF}^\cN_s$. Towards this end
		consider family of sets:	
		\begin{align*}
		{\cA}_s := \Big\{ A \cap C : \, & A \in \cG, C = \cap_{r=1}^n \cap_{l=1}^{p_r} \set{ N_{m_r}( (s_l^r , t_l^r ] \times D_l^r ) = k_l^r},  \\
		& 0  \leq s_{1}^r  < t_1^r \leq \ldots \leq s_{p_r}^r < t_{p_r}^r \leq s ,
		\ D^{r}_1, \ldots D^r_{p_r } \in \mathbfcal{X}, \
		k^r_1, \ldots , k^r_{p_r} \in \bN, \\
		& 0 \leq p_1 \leq \ldots \leq  p_r, \ 0 \leq m_1 \leq \ldots \leq  m_r, \quad
		r=1, \dots,n , \, n \in \bN \Big\}.
		\end{align*}
		Clearly, ${\cA}_s$ is a $\pi$-system and $\sigma( {\cA}_s) = \wh{\cF}^\cN_s$.
		Let us take $F \in {\cA}_s$, so $F = A \cap C $,
		and let $(s,t] \times D$ be disjoint with sets $(s_l^r , t_l^r ] \times D_l^r$ which define $C$.  This and $\cG$-conditional independence of  $\set{N_j}_{j\geq 1}$ imply that  events  $\set{N_j ((s,t] \times D)=k}$ and $C$ are conditionally independent given $\cG$. Hence,  by applying Lemma \ref{lem:cond-indep}, we obtain that \eqref{eq:DSMPP1} holds for $F \in {\cA}_s$.  Then, invoking the Monotone Class Theorem, we conclude that \eqref{eq:DSMPP1} holds for sets  $ F \in \wh{\cF}^\cN_s$. The proof is complete.
	\end{proof}

\subsection{Additional Technical Result}

	\begin{lemma}\label{lem:NCT}
		Let $(\mu^{k})_{ k=1}^\infty$ be a sequence of measures. Let $\mu$ be a mapping $\mu: \mathbfcal{X} \rightarrow [0,\infty]$ defined by
		\[
			\mu(A) = \lim_{n \rightarrow \infty} \sum_{k=1}^n \mu_k( A).
		\]
		Then $\mu$ is a measure.	Moreover for any measurable non negative  function $F: \cX \rightarrow \bR_+ $ we have
		\[
			\int_{\mathcal{X}} F d \mu
			=
			\lim_{n \rightarrow \infty}
			\sum_{k=1}^n
			\int_{\mathcal{X}} F d \mu_k
		\]
	\end{lemma}
	\begin{proof}
		The first part follows from the Nikodym convergence theorem (see e.g. Theorem 7.48 in Swartz \cite{Swa2009}).

		To prove the second assertion it suffices to consider simple step functions only, i.e.
		functions $F$ of the form
		\[
			F(x) :=
			\sum_{i=1}^n a_i \1_{A_i}(x), \quad  a_i \in \bR_+, A_i \in  \mathbfcal{X}.
		\]
		For such $F$ it holds
		\[
			\int_{\mathcal{X}} F d \mu
			=
			\sum_{i=1}^n a_i
			\mu(A_i)
			=
			\sum_{i=1}^n a_i
			\sum_{ k=1}^\infty
			\mu_k(A_i)
			=
			\sum_{ k=1}^\infty
			\sum_{i=1}^n a_i
			\mu_k(A_i)
			=
			\sum_{k=1}^\infty
			\int_{\mathcal{X}} F d \mu_k.
		\]
		Using usual approximation technique and the monotone convergence theorem we finish the proof.
	\end{proof}

\bibliographystyle{plain}
\bibliography{Math_Fin}
\end{document}